\documentclass[12pt]{amsart}
\usepackage{amsthm,amssymb,amsmath}
\usepackage{tikz-cd}
\usepackage[margin=1.1in]{geometry}
\usepackage{hyperref}

\newtheorem{lemma}{Lemma}
\newtheorem{prop}[lemma]{Proposition}
\newtheorem{proposition}[lemma]{Proposition}
\newtheorem{theorem}[lemma]{Theorem}
\theoremstyle{definition}
\newtheorem{remark}[lemma]{Remark}
\def\tgamma{\tilde \gamma}
\def\f{{\mathcal F}}
\def\af{{\rm af}}
\def\Gr{{\rm Gr}}
\def\Spec{{\rm Spec}}
\def\A{{\mathcal A}}
\def\la{\lambda}
\def\C{{\mathbb C}}
\def\R{{\mathbb R}}
\def\i{{\mathbf{i}}}
\def\j{{\mathbf{j}}}
\def\hw{{\mathrm{hw}}}
\def\Sym{{\mathrm{Sym}}}
\def\h{{\mathfrak h}}
\def\n{{\mathfrak n}}
\def\b{{\mathfrak b}}
\def\g{{\mathfrak g}}
\def\cG{{\mathring G}}
\def\pt{{\rm pt}}
\def\ep{{\varepsilon}}
\def\ph{{\varphi}}
\def\pr{{\rm pr}}
\def\ad{{\rm ad}}
\def\diag{{\rm diag}}
\def\Fl{{\mathcal Fl}}
\def\X{{\mathcal X}}
\def\Y{{\mathcal Y}}
\def\Z{{\mathbb Z}}
\def\op{{\rm op}}
\def\trop{{\rm trop}}
\def\O{{\mathcal O}}
\def\D{{\mathcal D}}
\def\F{{\mathcal F}}
\def\Re{{\rm Re}}
\def\L{{\mathcal L}}
\newcommand\ip[1]{\langle #1 \rangle}
\newcommand\defn[1]{{\it #1}}

\numberwithin{lemma}{section}
\begin{document}
\author{Thomas Lam}\address{Department of Mathematics, University of Michigan,
2074 East Hall, 530 Church Street, Ann Arbor, MI 48109-1043, USA}
\email{tfylam@umich.edu}\thanks{T.L. was supported by NSF grant DMS-1160726, and by a
Sloan Fellowship.}
\title{Whittaker functions, geometric crystals, and quantum Schubert calculus}
\begin{abstract}
This mostly expository article explores recent developments in the relations between the three objects in the title from an algebro-combinatorial perspective.  

We prove a formula for Whittaker functions of a real semisimple group as an integral over a geometric crystal in the sense of Berenstein-Kazhdan.  We explain the connections of this formula to the program of mirror symmetry of flag varieties developed by Givental and Rietsch; in particular, the integral formula proves the equivariant version of Rietsch's mirror symmetry conjecture. We also explain the idea that Whittaker functions should be thought of as geometric analogues of irreducible characters of finite-dimensional representations.
\end{abstract}
\maketitle

The heart of this article is a proof that certain integrals over geometric crystals are archimedean Whittaker functions, or equivalently, eigenfunctions of the quantum Toda lattice.  

A recent new development \cite{BBF1,BBF2} in the study of automorphic forms involves expressing multiple Dirichlet series and $p$-adic (metaplectic) Whittaker functions as sums over Kashiwara's crystals.  This motivated me to observe\footnote{Konni Rietsch has pointed out that this observation had also been made by Masaki Kashiwara.} that Rietsch's mirror-symmetric solution to the quantum Toda lattice \cite{Rie2} could be expressed as an integral over Berenstein and Kazhdan's geometric crystals \cite{BK1,BK2}.  This observation was also made in Chhaibi's recent thesis \cite{Chh} who developed a robust probabilistic interpretation of such integrals.

$$
\begin{tikzcd}
{ } & \text{Whittaker functions} \arrow{ld}[swap]{\text{integrals over}}  \arrow{rd}{\text{solutions to quantum $D$-module}}& \\
\text{Geometric crystals} \arrow[leftrightarrow]{rr}[swap]{\text{mirrors}} &  & \text{Flag variety}
\end{tikzcd}
$$

\section{Introduction}
\subsection{Whittaker functions}
The original Whittaker functions $\psi(z)$ are solutions to the {\it Whittaker differential equation}:
$$
\frac{d^2 \psi}{d z^2} + \frac{d \psi}{dz} + \left(\frac{k}{z} + \frac{\frac{1}{4}-m^2}{z^2} \right)\psi = 0.
$$
Jacquet then introduced Whittaker functions for reductive groups over local fields, with the original Whittaker function corresponding to the case of $SL_2(\R)$.  Jacquet's Whittaker functions play an important role in the study of automorphic forms and automorphic representations.

Kostant \cite{Kos:Whittaker} studied the Whittaker functions of real groups in detail, and essentially showed that they are the eigenfunctions of the quantum Toda lattice.  The quantum Toda lattice is a quantum integrable system with quantum Hamiltonian
$$
H =\frac{1}{2} \Delta - \sum_{i \in I}\alpha_i(t).
$$
We take this as the definition of Whittaker functions in our paper.

\subsection{Mirror symmetry for flag varieties}
For a sufficiently nice complex algebraic variety $M$ one can associate a quantum deformation $QH^*(M,\C)$ of the cohomology ring $H^*(M,\C)$.  The structure constants of this {\it quantum cohomology ring} are called {\it Gromov-Witten invariants}, and are obtained by enumerating rational curves in $M$.

Let $G$ be a complex semisimple algebraic group and $G/B$ be its flag variety.  Givental-Kim \cite{GK} initiated the study of the quantum cohomology ring $QH^*(G/B,\C)$, and Kim \cite{Kim} then proved that $QH^*(G/B,\C)$ was isomorphic to the ring of functions on the nilpotent leaf of the Toda lattice, the classical integrable system which is the quasi-clasical limit of the quantum Toda lattice.  These developments have led to a growing and vibrant subject of ``quantum Schubert calculus''.

At the heart of Givental's approach to quantum cohomology is an integrable system, called the quantum $D$-module.  The commutativity of this system of differential equations follows from deep, but general, results in Gromov-Witten theory.  Givental conjectured that solutions of the quantum $D$-module could be constructed as integrals over a (conjectural) mirror family.  

Kim showed that the quantum $D$-module of the flag variety is exactly the quantum Toda lattice.  So for the case of flag varieties, Givental's conjecture predicts integral formulae for Whittaker functions.  Givental \cite{Giv} explicitly constructed a mirror family and thus integral formula for $G = SL_n$ Whittaker functions.  This integral formula was then studied in detail by Gerasimov-Kharchev-Lebedev-Oblezin \cite{GKLO}, and a series of works by Gerasimov-Lebedev-Oblezin, including \cite{GLO,GLO1}, who extended the formula to classical groups.  

In \cite{Rie1}, Rietsch proposed a general Lie-theoretic construction of the mirror family, and proved it for the ``non-equivariant'' case in \cite{Rie2}.

\subsection{Geometric crystals}
Developments in the representation theory of quantum groups led Kashiwara to define a notion of a {\it crystal graph}: a combinatorial model $B$ for an irreducible representation $V$ of a complex simple Lie algebra $\g$.  The vertices of the graph represent basis vectors of $V$, while the action of $\g$ is encoded in colored edges.  

Berenstein and Zelevinsky \cite{BZ} studied Kashiwara's crystals by first parametrizing crystal basis elements using a string of integers, and then representing the crystal action as certain piecewise-linear formulae.  At the center of their approach is the relation between parametrizations of the closely related canonical bases and parametrizations of totally positive elements, observed by Lusztig \cite{Lus}.

Berenstein and Kazhdan \cite{BK1,BK2} then defined {\it geometric crystals}: these are irreducible varieties $Y$ equipped with rational actions $e_i: \C^* \times Y \to Y$, satisfying certain relations.  Berenstein and Kazhdan showed that geometric crystals tropicalize to combinatorial crystals: the rational actions $e_i$ tropicalize to the piecewise-linear crystal actions, and the finite set underlying a crystal is cut out of $\Z^\ell$ by the tropicalization of a {\it decoration} $\f: Y \to \C$.

\subsection{Whittaker functions as integrals over geometric crystals}
Let $G$ be a complex semisimple algebraic group with opposite Borels $B, B_-$, torus $T = B \cap B_-$ and let $U$ be the unipotent subgroup of $B$.  This article revolves around the following integral formula:
\begin{theorem}\label{thm:mainintro}
Let $\la: T \to \C^*$ be a character.  The integral
$$
\psi_\la(t) = \int_{(X_t)_{> 0}} \la(\gamma(x)) e^{-\f(x)}\omega_t 
$$
as a function of $t \in T_{>0}$ is a Whittaker function for $G$ with infinitesimal character $\xi_{\la - \rho}$.
\end{theorem}
Here $X_t = B_- \cap Ut \bar w_0U$ is the geometric crystal with highest weight $t$, the functions $\f: X_t \to \C$ and $\gamma: X_t \to T$ are the decoration and weight functions respectively, the form $\omega_t$ is a holomorphic form induced by a toric chart on $X_t$, and ``$>0$'' indicates totally positive parts.  The weight function $\gamma: X_t \to T$ tropicalizes to the usual weight function of a crystal.

In the mirror symmetry/quantum cohomology literature, the decoration is called the ``superpotential''.

We give an essentially self contained proof of Theorem \ref{thm:mainintro}, modulo issues of converegence of this integral, referring to \cite{Chh,Rie2} for such issues.  

Our proof follows that of Rietsch \cite{Rie2}, which in turn is motivated by work of Gerasimov-Kharchev-Lebedev-Oblezin \cite{GKLO}.  Our work improves on the existing literature in the following way.  Compared to Rietsch's work, we now allow $\la$ to be arbitrary: Rietsch deals with the case $\la = 1$.  In other words, we extend Rietsch's proof to the ``equivariant" case. We have also made an effort to formulate the proof within the theory of geometric crystals, hopefully clarifying many calculations.

Compared to Chhaibi's approach \cite{Chh}, our approach has the advantage that it proves that the integral is a solution to the whole quantum Toda lattice, rather than just the quantum Toda Hamiltonian.  We have also made an effort to make the proof as algebraic as possible, which may appeal to some readers.  Though I must say that the probabilistic interpretation of this integral is also extremely attractive!  Since Chhaibi's thesis is so comprehensive, we have basically omitted mention of the probabilistic viewpoint from our treatment.

One of the pleasant surprises (for me) is the naturality of the holomorphic top-form $\omega_t$.  When $X_t$ is embedded into the flag variety it can be identified with the open Richardson variety $R_{1,w_0} \subset G/B_-$.  The form $\omega_t$ is the (unique, up to scalar) meromorphic differential form on $G/B_-$ with simple poles exactly along the boundary $\partial R_{1,w_0}$, which is the union of Schubert and opposite Schubert divisors.  With Knutson and Speyer, we studied (the inverses of) such forms in the setting of Frobenius splittings \cite{KLS}.  This also indicates that the geometric crystal can be thought of as an open Calabi-Yau variety sitting in the flag variety.  It may thus be most natural to formulate Theorem \ref{thm:mainintro} as an integral over the flag variety, where the contribution of the integral over the boundary is zero.  We remark that the open Richardson varieties $R_{1,w_0}$ appear to play a special role in Lian-Song-Yau's study of period integrals on the flag variety where they are candidates for a {\it large complex structure limit point}; see Huang-Lian-Zhu \cite[Section 6]{HLZ}.

There are many potential generalizations to pursue, such as extensions to more general geometric crystals (especially the parabolic ones), but we have decided to focus on this one particular geometric crystal.

\subsection{Geometric analogues of Schur functions and geometric RSK}
The irreducible characters of $\g$ are weight generating functions of crystal graphs.  Theorem \ref{thm:mainintro} is thus the geometric analogue of this.  

For the case $G = SL_n$ this analogy is especially potent, and Whittaker functions behave like geometric analogues of Schur functions in many ways.  There is a geometric analogue of the Robinson-Schensted-Knuth bijection, introduced by Kirillov \cite{Kir}, and developed further by Noumi-Yamada \cite{NY}.  Instead of a bijection from matrices to tableaux, the geometric RSK bijection is a map $\R_{>0}^{nm} \to \R_{>0}^{nm}$ which tropicalizes to RSK.  

As explored in the works of Corwin, O'Connell, Sepp\"{a}l\"{a}inen, and Zygouras \cite{COSZ}, and O'Connell, Sepp\"{a}l\"{a}inen, and Zygouras \cite{OSZ}, the geometric RSK bijection is most interesting when suitable measures are put on $\R_{>0}^{nm}$.  This bijection then has applications to the theory of random directed polymers.  There are many works of O'Connell exploring these developements, see in particular \cite{O,O2}.  A recent ambitious generalization of these random processes is Borodin and Corwin's Macdonald processes \cite{BC}.

The geometric RSK bijection leads to a Cauchy identity for Whittaker functions \cite{COSZ}, and presumably also gives rise to the Pieri-like formula in \cite{GLO}.

\subsection{Organization}
In Section \ref{sec:background} we introduce some background notation, and also formulate some of Berenstein-Zelevinsky's work on parametrizations of subvarieties of complex semisimple algebraic groups.  In Section \ref{sec:geomcrystal} we give a condensed introduction to Berenstein-Kazhdan's geometric crystals.  For brevity, we have avoided unipotent crystals and unipotent bicrystals, but have given an example of how to tropicalize to obtain a combinatorial crystal.  In Section \ref{sec:whittaker}, we give a mostly self-contained introduction to the quantum Toda lattice, and explain how it arises from Whittaker modules.  Here I partly follow Etingof's article \cite{Eti}.  Section \ref{sec:integral} gives the proof of Theorem \ref{thm:mainintro}.  In Section \ref{sec:schur}, we write down the integral explicitly in type $A$ in terms of Gelfand-Tsetlin patterns.  We also state some Whittaker function identities which are geometric analogues of the Cauchy and Pieri identities for Schur functions.  In Section \ref{sec:mirror}, we connect the story back to quantum Schubert calculus, which is one of our (and Rietsch's) main motivations.

\medskip
{\bf Acknowledgements.}  
I have learnt a lot from Konni Rietsch on this topic.   I am also very grateful to Sasha Braverman and Reda Chhaibi for many helpful discussions while visiting ICERM.  Finally, we thank Ben Brubaker, Dan Bump, and Sol Friedberg for interesting discussions related to Whittaker functions.  We also thank Ben Brubaker, Ivan Corwin, and Viswambhara Makam for a number of comments on an earlier version.

This article is loosely based upon two talks I gave during the ICERM semester program ``Automorphic Forms, Combinatorial Representation Theory and Multiple Dirichlet Series".  I thank the organizers and ICERM for inviting me.    I am also grateful to the organizers of the International summer school and conference on Schubert calculus in Osaka 2012 for allowing me to write this article for the proceedings of the conference.

\section{Background on semisimple groups}\label{sec:background}
This section follows Berenstein and Zelevinsky \cite{BZ} and Berenstein and Kazhdan \cite{BK2}.
\subsection{Notations}

Let $G$ be a semisimple complex algebraic group with Dynkin diagram $I$.  We pick opposite Borels $B, B_-$ and set $T = B \cap B_-$.  Write $U$ and $U_-$ for the unipotent radicals of $B$ and $B_-$.  For each $i \in I$ we fix a homomorphism $\phi_i: SL_2 \to G$, and define
$$
x_i(a) = \phi_i \left(\begin{array}{cc} 1 & a \\0&1 \end{array}\right) \qquad y_i(a) = \phi_i\left(\begin{array}{cc} 1 & 0 \\a&1 \end{array}\right)  
$$
$$
\alpha_i^\vee(a) =  \phi_i\left(\begin{array}{cc} a & 0 \\0&a^{-1} \end{array}\right) \qquad x_{-i}(a) = \phi_i \left(\begin{array}{cc} a^{-1} & 0 \\1&a \end{array}\right)
$$
so that $x_i(a) \in B$, $y_i(a)\in B_-$, and $\alpha_i^\vee(a) \in T$.  They are related by the equality $x_{-i}(a) =y_i(a)\alpha^\vee_i(a^{-1})$.  For example for $G = SL_4$, we would have
$$
x_2(a) = \left(\begin{array}{cccc} 1 & && \\&1&a&\\&&1& \\&&&1\end{array}\right) 
\qquad \text{and} \qquad
x_{-3}(a) = \left(\begin{array}{cccc} 1 &&& \\&1&&\\&&a^{-1}& \\&&1&a\end{array}\right) 
$$

%\fixit{Need this?}Following \cite{BZ}, define the positive inverse anti-automorphism $\iota: G \to G$ by
%$$
%\iota(x_i(a)) = x_i(a) \qquad \iota(t) = t^{-1} \qquad \iota(y_i(a)) = y_i(a)
%$$
%for $i \in I$ and $t \in T$.

% Let $g \mapsto g^T$ denote the transpose anti-automorphism of $G$, given on generators by $x_i(a)^T = y_i(a)$ and $(\alpha_i^\vee(a))^T = \alpha_i^\vee(a)$.  

Define projections $\pi^+:B_- \cdot U \to U$ and $\pi^-:B_- \cdot U \to U$ by 
$$
\pi^+(bu) = u \qquad \pi^-(bu) = b.
$$

For each $i \in I$ we also have a character $\alpha_i: T \to \C^*$, which we will often think of as characters on $B$ or $B_-$ via the quotient maps $B \to B/U \simeq T$ and $B_- \to B_-/U_- \simeq T$.  We have $\alpha_j(\alpha_i^\vee(c)) = c^{a_{ij}}$ where $a_{ij} = \ip{\alpha_j,\alpha_i^\vee}$ is the Cartan matrix of $G$.  We denote the set of roots $\alpha$ of $T$ by $R^+$, and the corresponding coroots are denoted $\alpha^\vee$.

Also define $\bar s_i = x_i(-1)y_i(1)x_i(-1)$.  Since the $\bar s_i$ satisfy the braid relations, this gives a distinguished lifting of the $w \to \bar w$ of the Weyl group $W$ to $G$.  Since we have fixed such a lifting, we will often abuse notation by writing $w$ when we mean $\bar w$.  Denote by $w_0$ the longest element of $W$.

For each $i \in I$ we have an elementary character $\chi_i: U \to \C$ given by
$$
\chi_i(x_j(a)) = \delta_{ij} \cdot a.
$$
Similarly define $\chi_i^-: U_- \to \C$.  These characters are extended to rational functions on $G$ via
$$
\chi^-_i(u_-t u_+) = \chi_i^-(u_-) \qquad \text{and} \qquad \chi_i^+(u_-tu_+)=\chi_i(u_+)
$$
for $u_\pm \in U_\pm$ and $t \in T$.  We let $\chi = \sum_i \chi_i$.  For example, if $G = SL_4$
$$
\chi\left(\begin{array}{cccc} 1 &a &*&* \\&1&b&*\\&&1&c\\&&&1\end{array}\right) = a + b + c.
$$

\subsection{Relations for $x_i$ and $y_i$}
We recall some standard relations that the one-parameter subgroups $x_i(a)$ and $y_i(a)$ satisfy.

\begin{prop}\label{prop:xy}
We have
\begin{align}\label{eq:xy}
x_i(a) y_i(a') &= y_i\left(\frac{a'}{1+aa'}\right) \alpha_i^\vee(1+aa') x_i\left(\frac{a}{1+aa'}\right) \\
y_i(a) x_i(a') &= x_i\left(\frac{a'}{1+aa'}\right) \alpha_i^\vee\left(\frac{1}{1+aa'}\right) y_i\left(\frac{a}{1+aa'}\right)
\end{align}
and $x_i$ and $y_j$ commute for $i \neq j$.
\end{prop}

By a positive rational function in the following we mean a ratio of polynomials with nonnegative coefficients.

\begin{prop}[\cite{BZ}]\label{prop:xx}
Let $i \neq j \in I$ and suppose $(s_i s_j)^{m} = 1$ where $m = 2,3,4,6$.  Then as rational morphisms,
\begin{align*}
x_i(a)x_j(b) &= x_j(b) x_i(a) & \mbox{if $m = 2$}\\
x_i(a)x_j(b)x_i(c) &= x_j(bc/(a+c)) x_i(a+c) x_j(ab/(a+c)) & \mbox{if $m = 3$} \\
x_i(a_1)x_j(a_2)x_i(a_3)x_j(a_4) &= x_j(b_1) x_i(b_2)x_j(b_3) x_i(b_4) & \mbox{if $m = 4$} \\
x_i(a_1)x_j(a_2)x_i(a_3)x_j(a_4)x_i(a_5)x_j(a_6) &= x_j(c_1) x_i(c_2)x_j(c_3) x_i(c_4)x_j(c_5) x_i(c_6) & \mbox{if $m = 6$} 
\end{align*}
where $b_1,b_2,b_3,b_4$ (resp. $c_1,c_2,\ldots,c_6$) are positive rational functions of $a_1,a_2,a_3,a_4$ (resp. $a_1,a_2,\ldots,a_6$) depending only on the Cartan matrix entries $a_{ij}$ and $a_{ji}$.  Similar relations hold for $x_{-i}$, and for $m = 2,3$ they are
\begin{align*}
x_{-i}(a)x_{-j}(b) &= x_{-j}(b) x_{-i}(a) & \mbox{if $m = 2$}\\
x_{-i}(a)x_{-j}(b)x_{-i}(c) &= x_{-j}(bc/(b+ac)) x_{-i}(ac) x_{-j}(a+b/c) & \mbox{if $m = 3$}.
\end{align*}
\end{prop}
For a reduced word $\i = (i_1,i_2,\ldots,i_k)$, we denote $x_\i(a_1,a_2,\ldots,a_k) := x_{i_1}(a_1) \cdots x_{i_k}(a_k)$.  Proposition \ref{prop:xx} shows that if $\i$ and $\j$ are two reduced words for the same $w \in W$, then $x_\i(a_1,a_2,\ldots,a_k) = x_\j(a'_1,a'_2,\ldots,a'_k)$, for parameters $a'_j$ which are subtraction-free rational functions in the $a_j$-s.  Similarly we use the notation $x_{-\i}$.

\subsection{Toric charts}
We now discuss positive parametrizations of spaces by tori following \cite[Section 3]{BK2}.

Let $S = (\C^*)^r$ be an algebraic torus.  Let $X^*(S)$ be the group of characters of $S$.  A regular function $f$ on $S$ is {\it positive} if it is of the form $f(s)= \sum_{\lambda \in X^*(S)} a_\lambda \lambda(s)$ for nonnegative real numbers $a_\lambda$.  A rational function on $S$ is \defn{positive} if it can be expressed as a ratio of positive regular functions.  A rational function $f: S \to S'$ between two algebraic tori is positive if $\mu \circ f$ is positive for every character $\mu \in X^*(S')$.  

A {\it toric chart} on an algebraic variety $Y$ is a birational isomorphism $\theta:S \to Y$.  Two toric charts $\theta:S \to Y$ and $\theta':S' \to Y$ are {\it positively equivalent} if $(\theta)^{-1} \circ \theta'$ and $(\theta')^{-1} \circ \theta$ are both positive.  A {\it positive structure} $\Theta_Y$ on $Y$ is a positive equivalence class of toric charts and call the pair $(Y, \Theta_Y)$ a positive variety.  We define the {\it totally positive part} $Y_{>0}$ of $(Y,\Theta_Y)$ to be the image $\theta(\R_{>0}^r) \subset Y$ for any toric chart $\theta: S \to Y$ in the equivalence class $\Theta_Y$.  By the definitions, this does not depend on the choice of $\theta$.

The one-parameter subgroups $x_i$ and $y_i$ can be used to give toric charts of certain subvarieties of $G$.  Let $U^{w_0} = U \cap B_- w_0 B_-$ and let $X = B_- \cap Uw_0U$.  These two varieties are the main players in this article.

\begin{proposition}[{\cite[Proposition 4.5]{BZ}}]
Let $\i$ be a reduced word for $w_0$ and let $\ell = \ell(w_0)$.  The map $x_\i$ is an open embedding
$$
x_\i: (\C^*)^\ell \hookrightarrow U^{w_0}.
$$
For any reduced words $\i,\i'$ of $w_0$, the toric charts $x_\i$ and $x_{\i'}$ of $U^{w_0}$ are positively equivalent.  The map $x_{-\i}$ is an open embedding
$$
x_{-\i}: (\C^*)^\ell \hookrightarrow X.
$$
For any reduced words $\i,\i'$ of $w_0$, the toric charts $x_{-\i}$ and $x_{-\i'}$ of $X$ are positively equivalent.
\end{proposition}
Note that the ``positively equivalent'' statement follows from Proposition \ref{prop:xx}.  We denote the positive varieties by $(U^{w_0},\Theta^+)$ and $(X,\Theta^-)$.  We also have totally positive parts $U^{w_0}_{>0}$ and $X_{>0}$.

\begin{proposition}[{\cite[Lemma 2.13]{CA3}}] \label{prop:codim}
The union of the images of $x_\i$, as $\i$ varies over reduced words of $w_0$ has codimension 2 in $U^{w_0}$.
The union of the images of $x_{-\i}$, as $\i$ varies over reduced words of $w_0$ has codimension 2 in $X$.
\end{proposition}

\begin{proposition}[{\cite[Proposition 4.2]{BK1}}]\label{prop:eta}
The map $\eta: U^{w_0} \to X$ given by
$$
\eta(u) = \pi^-(u\overline{w_0})
$$
is an isomorphism of positive varieties.  The inverse map is given by
$$
\eta^{-1}(x) = \pi^+(\overline{w_0}x^{-1})^{-1}.
$$
\end{proposition}
The map $\eta$ is called the ``twist map''; it is studied extensively in \cite{BZ,FZ}.  It follows from Proposition \ref{prop:eta} that for $u \in U^{w_0}$ there is a unique $\tau(u) \in U$ such that $u \bar w_0 \tau(u) = \eta(u)$.  For a reduced word $\i = (i_1,i_2,\ldots,i_k)$, let $\i^{\op}=(i_k,\ldots,i_1)$ denote the reversed reduced word.  The following result is essentially \cite[Lemma 6.1]{BZ}; see also \cite[Proof of Theorem 4.1.20]{Chh}.
\begin{lemma}\label{lem:monom}
Let $\i = (i_1,\ldots,i_k)$ be a reduced word.  Then
$$
x_{\i}(a_1,\ldots,a_k) = \left(\prod_{j=1}^k \beta_{j}^\vee(a_j)\right) x_{-\i^{\op}}(b_k,\ldots,b_1)^T
$$
where $(a_1,\ldots,a_k)$ and $(b_1,\ldots,b_k)$ are related by an invertible monomial transformation, and $\beta_k = s_{i_1} \cdots s_{i_{k-1}} \alpha_{i_k}$.
\end{lemma}
\begin{proof}[(Sketch proof.)]
Let $g \mapsto g^T$ denote the transpose anti-automorphism of $G$, given on generators by $x_i(a)^T = y_i(a)$ and $(\alpha_i^\vee(a))^T = \alpha_i^\vee(a)$.  We have 
$$
x_{-\i^{\op}}^T = \left(y_{i_k}(b_k)\alpha_{i_k}^\vee(b_k^{-1}) \cdots y_{i_1}(b_1)\alpha_{i_1}^\vee(b_1^{-1})\right)^T = \alpha_{i_1}^\vee(b_1^{-1}) x_{i_1}(b_1) \cdots \alpha_{i_k}^\vee(b_k^{-1}) x_{i_k}(b_k).$$
Now push the torus factors through the $x_i$'s using the basic relation $t x_i(a) t^{-1} = x_i(\alpha_i(t) a)$, for $t \in T$.
\end{proof}

%\begin{lemma}
%Suppose 
%\end{lemma}
%\begin{proof}
%Let $u = x_\i(a_1,\ldots,a_\ell)$.  Write $u = t$
%\end{proof}

%x_{-\i}(a_1,\ldots,a_k)^T = \alpha_{i_1}^\vee(t_1^{-1}) \cdots \alpha_{i_k}^\vee(t_k^{-1}) \; x_{\i^{\op}}(t'_k,\ldots,t'_k)

%where $(t_1,t_2,\ldots,t_m)$ and $(t'_1,\ldots,t'_m)$ are related by a monomial transformation.
%$$
%t'_k = t_k \cdot \prod_{l < k} t_l^{a_{i_k,i_l}}
%$$
%and
%$$
%t_k = t'_k \prod_{l < k} t'_l^{\ip{\alpha_{i_k},\beta^\vee_l}}
%$$

\subsection{Canonical form}
Given a torus $S = \{(x_1,x_2,\ldots,x_r)\} \simeq (\C^*)^r$, we have a canonical top form
$$
\omega_S = \frac{dx_1}{x_1} \frac{dx_2}{x_2} \cdots \frac{dx_r}{x_r}. 
$$
Let $y_i = \prod_{j=1}^r x_i^{c_{ij}}$ be a monomial transformation of the torus.  Then
\begin{equation}\label{eq:monomial}
 \frac{dy_1}{y_1} \cdots \frac{dy_r}{y_r}= \det(C)   \frac{dx_1}{x_1} \cdots \frac{dx_r}{x_r}.
\end{equation}
The algebraic group automorphisms of $S$ are given by invertible monomial transformations, and in that case $C \in GL_n(\Z)$ is invertible, and thus $\det(C) = \pm 1$.  Thus $\omega_S$ is well-defined up to sign.

If $\theta: S \to Y$ is a toric chart, then the pushforward of $\omega_S$ gives a top form $\omega_\theta$ on $Y$.  We say that two toric charts $\theta:S \to Y$ and $\theta':S \to Y$ are \defn{canonically equivalent} if $\omega_\theta$ and $\omega_{\theta'}$ are equal up to sign.  A canonically positive variety is a triple $(Y, \Theta_Y, \omega_Y)$ where $\Theta_Y$ is a canonically positive equivalence class of toric charts and $\omega_Y$ is the induced top form (defined up to sign).  A birational morphism $f: Y \to Z$ of canonically positive varieties is \defn{canonically positive} if $f$ and $f^{-1}$ are both positive morphisms and $f$ sends $\omega_Y$ to $\omega_Y$.

%\begin{remark}
%We believe one can show that a positive morphism being a canonical equivalence should imply certain properties of the corresponding tropicalizations.
%\end{remark}

\begin{proposition}[\cite{Rie1,Chh}]\label{prop:canon}
For any reduced words $\i,\i'$ of $w_0$, the toric charts $x_\i$ and $x_{\i'}$ of $U^{w_0}$ are canonically equivalent.  
For any reduced words $\i,\i'$ of $w_0$, the toric charts $x_{-\i}$ and $x_{-\i'}$ of $X$ are canonically equivalent.
\end{proposition}
\begin{proof}
This can be verified directly by using the relations in Proposition \ref{prop:xx}.  This was done for $x_\i$ in \cite{Rie1} case-by-case.  The braid relations for $x_\i$ and $x_{-\i}$ in Proposition \ref{prop:xx} are related by inversion of parameters (\cite[Proposition 7.3(5)]{BZ}), so the statement for $x_{-\i}$ follows.  Let us illustrate what needs to be checked for the $m = 3$ relation for $x_{-\i}$.  Suppose $x_{-i}(a)x_{-j}(b)x_{-i}(c) = x_{-j}(e) x_{-i}(f) x_{-j}(g)$, where $e = bc/(b+ac)$, $f=ac$, and $g=a+b/c$.  We need to verify that 
$$
\frac{da}{a} \frac{db}{b} \frac{dc}{c} = \pm \frac{de}{e} \frac{df}{f} \frac{dg}{g}.
$$
%It is convenient to use logarithmic coordinates $A =\log(a)$ and so on, and the relations become
%$$
%E = B+C - \log(e^B + e^{A+C}) \qquad F = A + C \qquad G = \log(e^A + e^{B-C}).
%$$
The Jacobian matrix is
$$
\left(
\begin{array}{ccc}
 -\frac{c}{a (b+a c)^2} & \frac{c}{b (b+a c)^2} & \frac{b}{a c (b+a c)^2} \\
 \frac{1}{a b} & 0 & \frac{1}{b c} \\
 \frac{1}{a b c} & \frac{1}{a b c^2} & -\frac{1}{a c^3} \\
\end{array}
\right)
$$
which has determinant $\frac{efg}{abc}$.
\end{proof}

We will always use the canonically positive equivalence class of Proposition \ref{prop:canon} for $U^{w_0}$ and $X$, and generally omit this from the notation.  The canonical forms are denoted $\omega_U$ and $\omega_X$.  The following result was first explicitly observed by Chhaibi \cite{Chh}. 

\begin{proposition}[\cite{Chh}]\label{prop:canonicallyequiv}
The morphism $\eta: U^{w_0} \to X$ is a canonically positive isomorphism.
\end{proposition}

To summarize, $U^{w_0}$ and $X$ have totally positive parts $U^{w_0}_{>0}$ and $X_{>0}$, and canonical top forms $\omega_U$ and $\omega_X$ that are sent to each other via $\eta$ and $\eta^{-1}$.

In the next section we give another explanation for why $\omega_U$ and $(\eta^{-1})_*(\omega_X)$ must agree (at least up to a scalar).

\subsection{Direct geometric interpretation of $\omega_U$ and $\omega_X$}
We give a direct geometric interpretation for $\omega_U$.  By Proposition \ref{prop:codim} and Proposition \ref{prop:canon}, the form $\omega_U$ is well-defined and has no zeroes on the whole of $U^{w_0}$.  We can embed $U^{w_0}$ into the flag variety $G/B_-$ via the map $u \mapsto uB_-/B_-$.  The image of $U^{w_0}$ is the open Richardson variety $R_{1,w_0}$: it is the intersection of the open Schubert cell with the opposite open Schubert cell \cite{Rie2,KLS}.

\begin{lemma}\label{lem:boundary}
There is up to scalar only one meromorphic top-form $\omega$ on $G/B_-$, holomorphic and non-vanishing on $R_{1,w_0} \subset G/B_-$, and with only simple poles on the boundary $\partial R_{1,w_0} = G/B_- \setminus R_{1,w_0}$.
\end{lemma}
\begin{proof}
The canonical bundle of $G/B_-$ is the line bundle $L_{-2\rho} \simeq G \times_{B_-} \C_{2\rho}$, where $\C_{2\rho}$ is the one-dimensional representation of $B_-$ given by $b \mapsto \rho^2(b)$.  So the polar divisor of $\omega$ must have homology class $2\rho$, or equivalently, twice the sum of the Schubert divisors.  We have $\partial R_{1,w_0}= D_1 \cup D_2 \cup \cdots \cup D_r \cup D'_1 \cup \cdots \cup D'_r$, where the Schubert divisor $D_i$ and opposite Schubert divisor $D'_i$ are homologous.  Since $\omega$ only has simple poles, we see that it has a pole of order one on each $D_i$ or $D'_i$.  The ratio of any two such forms will then be a holomorphic function on $G/B_-$, which must be constant.
\end{proof}

By a dlog form on a smooth irreducible variety $Z$ we mean a form which is a wedge of forms $\frac{df}{f}$ where $f$ is a rational function $Z$.  We learnt the following proof from David Speyer.

\begin{lemma}\label{lem:simple}
A dlog form only has simple poles.
\end{lemma}
\begin{proof}
Let $\omega = \frac{df_1}{f_1} \frac{df_2}{f_2} \cdots \frac{df_k}{f_k}$, and let $D$ be an irreducible component of the polar divisor.  By a monomial transformation of $(f_1,f_2,\ldots,f_k)$ we may assume that only $f_1$ vanishes on $D$.  (By \eqref{eq:monomial} this changes $\omega$ by a constant.)  But then $f_1 = h^\alpha g$ where $h$ is the function cutting out $D$, and $g$  neither vanishes or has poles on $D$.  So 
$$
\frac{df_1}{f_1} = \alpha \frac{dh}{h} + \frac{dg}{g},$$
and it is clear that $\omega$ has a pole of order at most one along $D$.
\end{proof}

Combining Lemmas \ref{lem:boundary} and \ref{lem:simple},
\begin{proposition}
Any holomorphic, non-vanishing dlog form on $U^{w_0}$ is a scalar multiple of the canonical form $\omega_U$.  In particular $(\eta^{-1})_*(\omega_X)$ is a scalar multiple of $\omega_U$.
\end{proposition}

This form, or more precisely its inverse which is an anticanonical section, is studied for partial flag varieties in \cite{KLS}, in the setting of Frobenius splittings.

\subsection{Example}
Let us illustrate Proposition \ref{prop:canonicallyequiv} for the case $G=SL_3$.  Pick the reduced word $\i=(2,1,2)$.  Then 
$$
u:= x_\i(a,b,c) =\left(
\begin{array}{ccc}
 1 & b & b c \\
 0 & 1 & a+c \\
 0 & 0 & 1 \\
\end{array}
\right)
$$
Using 
$$
w_0 =\left(
\begin{array}{ccc}
 0 & 0 & 1 \\
 0 & -1 & 0 \\
 1 & 0 & 0 \\
\end{array}
\right)
$$
one computes that
$$
\eta(u) = \left(
\begin{array}{ccc}
 b c & 0 & 0 \\
 a+c & \frac{a}{c} & 0 \\
 1 & \frac{1}{c} & \frac{1}{a b} \\
\end{array}
\right)
\qquad \text{and} \qquad
\tau(u) =\left(
\begin{array}{ccc}
 1 & \frac{1}{c} & \frac{1}{a b} \\
 0 & 1 & \frac{a+c}{a b} \\
 0 & 0 & 1 \\
\end{array}
\right).
$$
Note that both $u$ and $\eta(u)$ are totally nonnegative in the usual sense (all minors are nonnegative) when $(a,b,c) \in \R_{>0}^3$.  We also compute that that $\eta(u) = y_1(e) y_2(f) y_1(g)$,
where $e = 1/b$, $f = 1/ab$, and $g = 1/c$.  Since the map $(a,b,c) \mapsto (e,f,g)$ is an invertible monomial transformation which sends $\R_{>0}^3 \to \R_{>0}^3$ we see that $\eta$ sends $U^{w_0}_{>0}$ to $X_{>0}$ and that $$
\frac{da}{a} \frac{db}{b} \frac{dc}{c} = \pm \frac{de}{e} \frac{df}{f} \frac{dg}{g}.
$$

\section{Geometric crystals}
\label{sec:geomcrystal}
We follow \cite{BK2} to introduce geometric crystals.  For reasons of brevity we have chosen to omit the notions of unipotent bicrystals and unipotent crystals.
\subsection{Decorated geometric crystals}
A \defn{geometric crystal} is a $5$-tuple $(Y,\gamma,\ph_i,\ep_i,e_i \mid i \in I)$ where
\begin{enumerate}
\item
$Y$ is an irreducible algebraic variety over $\C$
\item 
$\gamma: Y \to T$ and $\ep_i, \ph_i: Y \to \C$ are rational maps
\item
$e_i: \C^* \times Y \to Y$ is a rational action of $\C^*$ on $Y$, where we write $e_i^c(y) = e_i(c,y)$.
\end{enumerate}
satisfying 
\begin{align*}
\alpha_i(\gamma(y)) &= \ep_i(y)/\ph_i(y), \qquad \gamma(e_i^c(y)) = \alpha_i^\vee(c) \gamma(y) \\
\ep_i(e_i^c(y)) &= c\ep_i(y), \qquad \ph_i(e_i^c(y)) =c^{-1} \ph_i(y)
\end{align*}
and the geometric braid relations
\begin{align*}
e_i^{c_1}  e_j^{c_2} &= e_j^{c_2}  e_i^{c_1}  & \mbox{if $\ip{\alpha_i,\alpha_j^\vee} = 0$} \\
e_i^{c_1}  e_j^{c_1c_2}  e_i^{c_2} &=e_j^{c_2}  e_i^{c_1c_2}  e_j^{c_1} & \mbox{if $\ip{\alpha_j,\alpha_i^\vee} = \ip{\alpha_i,\alpha_j^\vee} = -1$}\\
e_i^{c_1}  e_j^{c_1^2c_2}  e_i^{c_1c_2} e_j^{c_2} &=e_j^{c_2}  e_i^{c_1c_2}  e_j^{c_1^2c_2}  e_i^{c_1} & \mbox{if $\ip{\alpha_j,\alpha_i^\vee} = 2\ip{\alpha_i,\alpha_j^\vee} = -2$} \\
e_i^{c_1}  e_j^{c_1^3c_2}  e_i^{c_1^2c_2}  e_j^{c_1^3c_2^2}  e_i^{c_1c_2}  e_j^{c_2} &=e_j^{c_2}  e_i^{c_1c_2}  e_j^{c_1^3c_2^2}   e_i^{c_1^2c_2}  e_j^{c_1^3c_2}  e_j^{c_1} & \mbox{if $\ip{\alpha_j,\alpha_i^\vee} = 3\ip{\alpha_i,\alpha_j^\vee} = -3$}
\end{align*}
%\begin{align*}
%e_i^{c_1} \circ e_j^{c_2} &= e_j^{c_2} \circ e_i^{c_1}  & \mbox{if $\ip{\alpha_i,\alpha_j^\vee} = 0$} \\
%e_i^{c_1} \circ e_j^{c_1c_2} \circ e_i^{c_2} &=e_j^{c_2} \circ e_i^{c_1c_2} \circ e_j^{c_1} & \mbox{if $\ip{\alpha_i,\alpha_j^\vee} = \ip{\alpha_i,\alpha_j^\vee} = -1$}\\
%e_i^{c_1} \circ e_j^{c_1^2c_2} \circ e_i^{c_1c_2}\circ e_j^{c_2} &=e_j^{c_2} \circ e_i^{c_1c_2} \circ e_j^{c_1^2c_2} \circ e_i^{c_1} & \mbox{if $\ip{\alpha_j,\alpha_i^\vee} = 2\ip{\alpha_i,\alpha_j^\vee} = -2$} \\
%e_i^{c_1} \circ e_j^{c_1^3c_2} \circ e_i^{c_1^2c_2} \circ e_j^{c_1^3c_2^2} \circ e_i^{c_1c_2} \circ e_j^{c_2} &=e_j^{c_2} \circ e_i^{c_1c_2} \circ e_j^{c_1^3c_2^2}  \circ e_i^{c_1^2c_2} \circ e_j^{c_1^3c_2} \circ e_j^{c_1} & \mbox{if $\ip{\alpha_j,\alpha_i^\vee} = 3\ip{\alpha_i,\alpha_j^\vee} = -3$}
%\end{align*}

In \cite{BK2}, the rational functions $\ph_i,\ep_i$ are allowed to be zero, but for our purposes, it is simpler to always assume that $\ph_i, \ep_i$ are non-zero.  

A \defn{decorated geometric crystal} is a geometric crystal $(Y,\gamma,\ph_i,\ep_i,e_i \mid i \in I)$ equipped with a rational function $\f:Y \to T$, called the \defn{decoration}, which satisfies
\begin{equation}\label{eq:dec}
\f(e_i^c(y)) = \f(y) + \frac{c-1}{\ph_i(y)}+\frac{c^{-1}-1}{\ep_i(y)}.
\end{equation}

\subsection{The geometric crystal with highest weight $t$}
Let $\X = T \cdot X = B_- \cap UT w_0U$.  Define the \defn{highest weight map} $\hw:\X \to T$ by 
$$
\hw(u_1 t \bar w_0 u_2) = t
$$
and for $t \in T$, let $X_t = \hw^{-1}(t) = t \cdot X =X\cdot (w_0 \cdot t)$.  We shall call $X_t$ the \defn{geometric crystal with highest weight $t$.}  Clearly all the fibers $X_t$ are isomorphic.

Define \begin{enumerate}
\item the weight morphism $\gamma:X \to T$ by
$$
\gamma(x) = x \mod U_- \in B_-/U_- \simeq T;
$$
\item
$$
\ph_i(x) = \chi_i^-(x) \qquad \ep_i(x) = \ph_i(x)\alpha_i(\gamma(x));
$$
\item
$$
e_i^c(x) = x_i\left(\frac{c-1}{\ph_i(x)}\right) \cdot x \cdot x_i\left(\frac{c^{-1}-1}{\ep_i(x)}\right);
$$
\item
$$
\f(u_1tw_0u_2) = \chi(u_1) + \chi(u_2).
$$
\end{enumerate}
The fact that $e_i^c(x) \in \X$ follows from Lemma \ref{lem:Uaction} below.  There is a natural projection $\pr_i: B_- \to B_- \cap \phi_i(SL_2)$, and if 
$$
\pr_i(x) = \left( \begin{array}{cc} a & 0 \\ b & a^{-1} \end{array}\right)
$$
then $\ph(x) = b/a$ and $\ep(x) = ab$.

The first statement of the following is due to Berenstein and Kazhdan \cite{BK2}.  The second statement is due to Chhaibi \cite{Chh}.
\begin{theorem}\label{thm:geomcrystal}
Each $(X_t,\gamma,\ph_i,\ep_i,e_i, \f)$ is a decorated geometric crystal.  Furthermore, the geometric crystal action $e_i^c$ preserves the canonical volume form $\omega_t$.
\end{theorem}

There is a rational $U$-action on $X$ given by
$$
u(x):=u \cdot x \cdot \pi^+(u\cdot x)^{-1}
$$
where $\pi^+:B_- \cdot U \to U$ is the projection onto the right factor.

\begin{lemma}[{\cite[Lemma 2.4]{BK1}}]\label{lem:Uaction}
This action is given explicitly by
$$
(x_i(a))(x) = x_i(a) \cdot x \cdot x_i(a')
$$
where 
$$
a'=-\frac{a}{(1+a\ph_i(x))\alpha_i(\gamma(x))}.
$$
Furthermore, we have
$$
\gamma(x_i(a)(x)) =  \alpha_i^\vee(1+a\ph_i(x)) \gamma(x)
\qquad
\text{and}
\qquad
\frac{1}{\ph_i(x_i(a)(x))} = \frac{1}{\ph_i(x)} + a.
$$
\end{lemma}
\begin{proof}
Let $x = ut$ for $u \in U_-$ and $t \in T$.  Supposing that $u = y_{i_1}(a_1) \cdots y_{i_\ell}(a_\ell)$, the right factor $\pi^+(x_i(a) \cdot x)^{-1}$ can be computed using Proposition \eqref{prop:xy} and it is clear that it is of the form $x_i(a')$ for some $a'$.  To compute $a'$ it is enough to suppose $G = SL_2$ and let $x =\left(\begin{array}{cc} b & 0 \\ c & b^{-1} \end{array}\right)$.  Then
$$
\left(\begin{array}{cc} 1 & a \\ 0 & 1 \end{array}\right)\left(\begin{array}{cc} b & 0 \\ c & b^{-1} \end{array}\right) \left(\begin{array}{cc} 1 & a' \\ 0 & 1 \end{array}\right) = \left(\begin{array}{cc} b+ac & a'(b+ac)+ab^{-1} \\ c & b^{-1}+a'c \end{array}\right)
$$
so that $a' = -\frac{ab^{-1}}{b+ac}$.  But
$$
\ph_i\left(\left(\begin{array}{cc} b & 0 \\ c & b^{-1} \end{array}\right)\right) = \frac{c}{b} \qquad \alpha \circ \gamma\left(\left(\begin{array}{cc} b & 0 \\ c & b^{-1} \end{array}\right)\right) = b^2.
$$
The value $a'$ can also be directly computed using \eqref{eq:xy}.
% (see the proof of Lemma \ref{lem:U-action} for an analogous computation).  
Similarly, $1+a\ph_i(x) = 1+ac/b$, so that
$$
\gamma(x_i(a)(x)) = \alpha_i^\vee(1+a\ph_i(x))) \gamma(x).
$$
Also $\ph_i(x_i(a)(x)) = c/(b+ac)$, satisfying the claimed equality.
%or directly from \eqref{eq:xy}, we have
%$$
%x_i(a)(u) = u'\alpha_i^\vee(1+a\chi_i(u))) t  x_i(-a')
%$$
%so that
\end{proof}

\begin{proof}[Sketch proof of Theorem \ref{thm:geomcrystal}]
It  follows from Lemma \ref{lem:Uaction} and the definition that the $\C^*$-action $e_i$ preserves $X_t$.  Substituting the definition of $e_i$ into Lemma \ref{lem:Uaction}, we see that all the relations of a geometric crystal are satisfied except possibly the geometric braid relations relating $e^c_i$ and $e^{c'}_j$.  We omit the proof of these relations, which are verified in \cite{BK1}.

Finally, we have $\chi(x_i(a)u) = a +\chi(u)$, so that \eqref{eq:dec} just follows from the definition of $f$ and $e_i$.

The final statement follows from the fact that if $u = x_\i(a_1,a_2,\ldots,a_\ell)$ with $i = i_\ell$ then $\eta^{-1} \circ e_i^c \circ \eta(u) = x_\i(ca_1,a_2,\ldots,a_\ell)$ which preserves the logarithmic volume form. See \cite{Chh}.
\end{proof}

\subsection{Weight map in coordinates}
The following explicit formula for $\gamma:X \to T$ will be helpful.  

\begin{prop}\label{prop:gamma}
Suppose $u = x_\i(a_1,a_2,\ldots,a_\ell) \in U^{w_0}$.  Then
$$
\gamma(\eta(u)) = \prod_{k=1}^\ell \beta_{k}^\vee(a_k)
$$
where $\beta_k = s_{i_1} \cdots s_{i_{k-1}}\alpha_{i_k}$.
\end{prop}
\begin{proof}
Use Lemma \ref{lem:monom} to write $u = tb$ where $t = \prod_{k=1}^\ell \beta_{k}^\vee(a_k)$ and $b = x_{-\i^{\op}}(b_k,\ldots,b_1)^T$.  Since $b \in B \cap U_- w_0 U_-$, we have that $bw_0 \in U_- \cdot U$, thus 
\begin{equation*}
\gamma(\eta(u)) = \gamma(\pi_-(uw_0)) = \gamma(\pi_-(tbw_0)) = t. \qedhere
\end{equation*}
%For $x = x_{-\i}(b_1,b_2,\ldots,b_\ell)$, $\gamma(x) = \prod_{k=1}^\ell \alpha_{i_k}^\vee(b_k)^{-1}$.  For an appropriate $\j$, we have that $x_{\j}^{-1} \eta^{-1}(x)$ is related to $(b_1,\ldots,b_\ell)$ by an invertible monomial transformation.
\end{proof}
In the notation we use, this is a special case of \cite[Theorem 4.1.20]{Chh}.  It is also a slight variant of \cite[Claim 7.12]{BK2} or \cite[(6.3)]{BZ}.

\subsection{The positive decorated geometric crystal}
We equip $\X$ with the toric chart $T \times S \to \X$, $(t,s) \mapsto t \cdot \theta(s)$ where $\theta:S \to X$ is a toric chart in $\theta_X$.  Fixing $t \in T$ gives toric charts on each of the highest weight geometric crystals $X_t$.  In particular, on each $X_t$ we have a canonical form $\omega_t$, and a totally positive part $(X_t)_{> 0} = t \cdot X_{>0}$.  Note that if $t$ itself is not totally positive, $t \cdot X_{>0}$ will not be totally positive in the usual sense.

Suppose $\Y =  (Y,\gamma,\ph_i,\ep_i,e_i,\f)$ is a decorated geometric crystal and $\Theta$ is a positive structure on the variety $Y$.  Then $(\Y,\Theta)$ is a positive decorated geometric crystal if
\begin{enumerate}
\item
the morphism $\gamma:Y \to T$ is a morphism\footnote{The definition of a morphism of positive varieties in \cite{BK2} has some subtleties, and we do not introduce the full definitions here.} of positive varieties $(Y,\Theta) \to (T, \Theta_T)$;
\item
the functions $\ph_i,\ep_i, \f$ are $\Theta$-positive;
\item
the map $e_i: \C^* \times Y \to Y$ is a morphism of positive varieties $(\C^* \times Y,\Theta_{\C^*} \times \Theta) \to (Y, \Theta_Y)$.
\end{enumerate}

\begin{theorem}\label{thm:posdec}
The decorated highest weight crystal $X_t$, equipped with the positive structure $\Theta_t$ is a positive decorated geometric crystal.
\end{theorem}

Each of the maps $\ph_i,\ep_i,\f$ can be expressed in terms of the generalized minors of \cite{BZ,FZ}.  Theorem \ref{thm:posdec} essentially follows from the subtraction-free rational expressions (see \cite{BZ,FZ}) for these minors in terms of the parameters in $(a_1,\ldots,a_\ell)$ of $x_\i(a_1,\ldots,a_\ell)$.

For positive decorated geometric crystals, Berenstein and Kazhdan define a tropicalization functor which produces a (Kashiwara) combinatorial crystal.  

\subsection{Combinatorial crystals from geometric crystals}\label{sec:combcrystal}
In this section, we assume the reader is familiar with Kashiwara's crystals; see \cite{BK2} for further details.  We do not describe Berenstein and Kazhdan's tropicalization procedure formally, but expain how it works in the case $G = GL_4$.  Strictly speaking we should be using the semisimple group $SL_4$, but the coordinates are easier to write for $GL_4$.  We fix a reduced word $\i = (3,2,1,3,2,3)$ of $w_0$, and consider the toric chart $\theta: T \times (\C^*)^6 \to \X$ given by 
$$
(t_1,t_2,t_3,t_4) \times (a,b,c,d,e,f) \longmapsto x_{\i}(a,b,c,d,e,f) t u' = x
$$
where $t = \diag(t_1,t_2,t_3,t_4)$, and $u'$ is chosen so that the $x_{\i}(a,b,c,d,e,f) t w_0 u'$ lies in $B_-$.  We get that
$$
u' = \left(
\begin{array}{cccc}
 1 & \frac{t_3}{f t_4} & \frac{t_2}{d e t_4} & \frac{t_1}{a b c t_4} \\
 0 & 1 & \frac{(d+f) t_2}{d e t_3} & \frac{(a+d+f) t_1}{a b c t_3} \\
 0 & 0 & 1 & \frac{(d e+a b+ae) t_1}{a b c t_2} \\
 0 & 0 & 0 & 1 \\
\end{array}
\right) \qquad x = \left(
\begin{array}{cccc}
 c e f t_4 & 0 & 0 & 0 \\
 (e f+b d+bf) t_4 & \frac{b d t_3}{f} & 0 & 0 \\
 (a+d+f) t_4 & \frac{(a+d) t_3}{f} & \frac{a t_2}{d e} & 0 \\
 t_4 & \frac{t_3}{f} & \frac{t_2}{d e} & \frac{t_1}{a b c} \\
\end{array}
\right)
$$
so that
$$
\f(x) = a+b+c+d+e+f + \frac{t_3}{f t_4} + \frac{(d+f) t_2}{d e t_3} + \frac{(d e+a b+ae) t_1}{a b c t_2}.
$$
(See Section \ref{sec:gelfand} for a general formula for this.)  Let $A,B,C,D,E,F,T_1,T_2,T_3,T_4$ be the ``tropicalizations'' of the ten variables.  Fix integers $T_1,T_2,T_3,T_4 \in \Z^4$, representing a highest weight vector.  Our first step is to consider the integer 6-tuples $A,B,C,D,E,F$ satisfying $\trop(\f) \geq 0$.  Here the tropicalization $\trop(\f)$ is obtained by the the substitution $(+,\times,\div) \mapsto (\min,+,-)$ and changing variables to their tropicalizations.  We get
\begin{align*}
\trop(\f) &= \min(A,B,C,D,E,F, T_3-F-T_4,T_2-E-T_3,T_2+F-D-E-T_3,
\\
&T_1-C-T_2,T_1+E-B-C-T_2,T_1+D+E-A-B-C-T_2) \geq 0.
\end{align*}
Thus the underlying set of the combinatorial crystal is the set of $(A,B,C,D,E,F) \in \Z^6$ satisfying the above inequality.
These variables can be arranged into a Gelfand-Tsetlin pattern
$$
\begin{array}{ccccccc}T_1 && T_2 &&T_3 &&T_4  \\
 & T_1-C && T_2-E && T_3-F& \\
 && T_1-C-B && T_2-D-E && \\
&&& T_1-A-B-C &&&
 \end{array}
$$
with the usual inequalities.  The weight of such a pattern is $\trop(\gamma)$ where
$$
\gamma = \left(c e f t_4,\frac{b d t_3}{f},  \frac{a t_2}{d e} , \frac{t_1}{a b c}\right),
$$
agreeing with the usual weight.   Let us also calculate a single crystal operation.  We have $\ph_2(x) = (a+d)/bd$ and $\ep_2(x) = (de(a + d) t_3)/(a f t_2)$ and 
$$
e^p_2(x) = \left(
\begin{array}{cccc}
 c e f t_4 & 0 & 0 & 0 \\
 \frac{(d (e f+b dp+bfp)+a (e f+b f+bd p)) t_4}{a+d} & \frac{b d p t_3}{f} & 0 & 0 \\
 (a+d+f) t_4 & \frac{(a+d) t_3}{f} & \frac{a t_2}{d e p} & 0 \\
 t_4 & \frac{t_3}{f} & \frac{a t_2+d p t_2}{e p d^2+a e p d} & \frac{t_1}{a b c} \\
\end{array}
\right).
$$
Solving gives that the coordinates of $e^p_2(x)$ are given by
$$
a'= (a^2 + a d)/(a + d p), b'=(a b + b d p)/(a + d), c'=c, d' =(d(a + d) p)/(a + d p), e'=e,f'=f.
$$
To obtain the combinatorial crystal action, we tropicalize this, and set $P = 1$.  Thus, for example the new value $D'$ of $D$ is equal to 
$$
D' = \min(A+D+1,2D+1)-\min(A,D+1).
$$

\section{Whittaker functions and Whittaker modules}\label{sec:whittaker}
This section gives a condensed introduction to the quantum Toda lattice, partly following the approach of Etingof \cite{Eti}.  We will return to a discussion of the classical Toda lattice in Section \ref{sec:classicaltoda}.
\subsection{Quantum Toda lattice}
For our purposes, a Whittaker function $\psi \in C^\infty(T)$ will be a smooth function on $T$ which is an eigenvector of the quantum Toda lattice, a system of commuting differential operators, the first one being:
$$
H =\frac{1}{2} \Delta - \sum_{i \in I}\alpha_i(t)
$$
where $\Delta$ is the Laplacian associated to the $W$-invariant inner product on $\h$.  Explicitly, $\Delta = \sum_i (\frac{\partial}{\partial h_i})^2$ where $h_i$ are an orthonormal basis of $\h$.  To be more precise, for $X \in \h$, the differential operator $\frac{\partial}{\partial X}$ acts on $C^\infty(T)$ as
$$
\frac{\partial f}{\partial X} (t) = \left.\frac{d}{da}f(t\exp(aX))\right|_{a=0}
$$

\begin{theorem}\label{thm:Toda}
There exists a unique set of differential operators $H_1,H_2,\ldots,H_{r}$ on $T$, called the quantum Toda lattice, such that
\begin{enumerate}
\item
the operators $H_i$ commute;
\item
we have $H_1 = H$; and
\item
the symbols of $H_i$ are the fundamental $W$-invariants of $\Sym(\h)$.
\end{enumerate}
\end{theorem}

Kazhdan and Kostant constructed this commuting set of differential operators by quantum Hamiltonian reduction from the action of $Z(\g)$ on $C^\infty(G)$.  The existence part of Theorem \ref{thm:Toda} is explained in Proposition \ref{P:Toda} below.  For $G = SL_n$, the quantum Toda lattice are essentially the coefficients of the characteristic polynomial of the following matrix:
\begin{equation}\label{eq:jacobi}
 \left(\begin{array}{ccccc}
\frac{\partial}{\partial h_1}&e^{h_2-h_1}&&\vphantom{\ddots} \\
1&\frac{\partial}{\partial h_2}&e^{h_3-h_2}&\vphantom{\ddots} \\
&1&\frac{\partial}{\partial h_3}&\ddots \\
&&\ddots&\ddots&e^{h_n-h_{n-1}} \\
&&&1&\frac{\partial}{\partial h_n}
\end{array}\right)
\end{equation}
where $h_i$ are standard coordinates on $\h$, satisfying $\sum_i h_i = 0$.  The characteristic polynomial is well-defined even though the matrix has non-commutative entries.

\subsection{Center and Harish-Chandra homomorphism}  
The reader is referred to \cite{Hum} for further details of the material in this section.

For each $\alpha \in R^+$ pick weight vectors $e_\alpha,f_\alpha,h_\alpha \in \g$ satisfying $[e_\alpha,f_\alpha] = h_\alpha$.  As usual we write $e_i = e_{\alpha_i}$ and so on.  We assume that $x_i(a) = \exp(ae_i)$, $y_i(a) =  \exp(af_i)$ and $\alpha_i^\vee(a) = \exp(ah_i)$.

We recall some standard facts concerning the center $Z(\g) \subset U(\g)$ of the universal enveloping algebra.  We have a decomposition $U(\g) = U(\h) \oplus (\n_-U(\g)+ U(\g)\n_+)$.  This gives a projection map $U(\g) \to U(\h)$.  Restricted to $Z(\g)$, this projection is an algebra homomorphism $HC': Z(\g) \to U(\h)$.  We may identify $U(\h)$ with the ring of polynomial functions $\O(\h^*)$ on $\h^*$.  Composing $HC'$ with the algebra automorphism $p(\lambda) \mapsto p(\lambda-\rho)$ induces the Harish-Chandra isomorphism
$$
HC: Z(\g) \longmapsto \O(\h^*)^W,
$$
where $\rho = \frac{1}{2} \sum_{\alpha \in R^+} \alpha$ is the half-sum of positive roots.

Suppose that $z \in Z(\g)$.  Note that if we write $z = x + y$ where $x \in \n_- U(\g)$ and $y \in U(\b_+)$, then automatically $y = HC'(z)$.  This follows easily from the fact that $Z(\g)$ lies in the centralizer of $\h$ in $U(\g)$, and thus every element of $Z(\g)$ has weight 0.

Now let $\xi_\la: Z(\g) \to \C$ be the central character of $Z(\g)$ by which $Z(\g)$ acts on the highest weight irreducible representation of $\g$ with highest weight $\la \in \h^*$.  The Harish-Chandra isomorphism satisfies
$$
\xi_\la(z) = HC(z)(\la+\rho) \qquad \mbox{for all $z \in Z(\g)$.}
$$
Note that $\xi_\la = \xi_{w \cdot \la}$ where as usual the dotted action is given by $w \cdot \la = w(\la+\rho)-\rho$.  We also define $\rho^\vee = \frac{1}{2} \sum_{\alpha \in R^+} \alpha^\vee$.

If $V$ is a $\g$-module then so is $V^*$ under the action $(X \cdot v^*)(v) = v^*(-X \cdot v)$ for $v \in V, v^* \in V^*$.  Note that if $V_\la$ is finite-dimensional then $(V_\la)^* \simeq V_{-w_0 \la}$.  It follows that if $V$ has central character $\xi_\la$, then $V^*$ has central character $\xi_{-w_0\la}$.

Let $(\cdot,\cdot)$ be the Killing form of $\g$ which restricts to a nondegenerate symmetric bilinear form on $\h$.  Using $(\cdot,\cdot)$ we can identify $\h$ with $\h^*$, and we also use $(\cdot,\cdot)$ to denote the corresponding form on $\h^*$.  Define the \defn{Casimir element} $C \in Z(\g) \subset U(\g)$ by
$$
C = \sum_i a_i b_i 
$$
where $\{a_i\}$ and $\{b_i\}$ are dual bases with respect to the nondegenerate symmetric bilinear form $(\cdot,\cdot)$.  It can be expressed in terms of Chevalley generators as
$$
C = \sum_{i=1}^r h_i^2 + 2\sum_{\alpha \in R_+}f_\alpha e_\alpha + 2h_{\rho^\vee}
$$
where $h_i$ is an orthonormal basis of $\h$ and $2h_{\rho^\vee} = \sum_{\alpha \in R^+} h_{\alpha}$.
%where we have used that $\rho^\vee$ satisfies $\ip{\alpha_i, \rho^\vee} = 1$.  
To obtain this formula, we used the commutation relation $[e_\alpha,f_\alpha] = \alpha^\vee$. 
Note that $HC'(C) = \sum_{i=1}^r h_i^2 + 2h_{\rho^\vee}$ which can be identified with the function $\la \mapsto (\la,\la+2\rho)$ on $\h^*$.  Thus we have 
\begin{equation}\label{eq:HCC}
HC(C) (\la) = (\la,\la) -(\rho,\rho).
\end{equation}

\subsection{Whittaker modules}

Let $V$ be a $(\g,T)$-module.  That is, $V$ is a complex vector space with compatible actions of $U(\g)$ and of $T$: if $X \in \h$, we have 
\begin{equation}\label{eq:gT1}
X\cdot v =  \left.\frac{d}{ds}\exp(aX) \cdot v\right|_{a=0}
\end{equation}
and for $X \in \g$ we have
$$
t\cdot X \cdot t^{-1} \cdot v = \ad(t)(X) \cdot v
$$
where $\ad(t):\g \to \g$ is the adjoint action of $T$ on $\g$.  We will mostly use the compatibility of the $\g$ and $T$ actions in a formal way, ignoring topological considerations.

%$t = \exp(X)$ then $t.v = \exp(X).v = v +X.v +\frac{X^2}{2}.v + \cdots$.
%\footnote{The change in sign is to match a convention later on, and is not significant.} 

A vector $v \in V$ is a \defn{$\n_+$-Whittaker vector} (resp. \defn{$\n_-$-Whittaker vector}) if $e_i \cdot v = -v$ (resp. $f_i \cdot v = -v$ ) for $i \in I$.  If $V_\pm$ are two $(\g,T)$-modules, we say that a bilinear pairing $\ip{\cdot,\cdot}: V_- \times V_+ \to \C$ is $\g$-invariant if we have
$$
\ip{X\cdot v, w} + \ip{v, X \cdot w} = 0
$$
for all $X \in \g$.  

\begin{prop}\label{P:Toda}
Let $\ip{.,.}: V_- \times V_+ \to \C$ be a $\g$-invariant bilinear pairing.  Suppose $v_+ \in V_+$ and $v_- \in V_-$ are $\n_+$- and $\n_-$-Whittaker vectors respectively.  Suppose that $V_+$ has central character $\xi$.  Then the function $\psi \in C^\infty(T)$ given by
$$
\psi: t \longmapsto \rho^{-1}(t)\ip{v_-, t^{-1} \cdot v_+}
$$
is a Whittaker function.
\end{prop}
\begin{proof}
Let $\psi'(t) = \ip{v_-, t^{-1} \cdot v_+}$.  Let $C$ be the Casimir element.  We calculate the  function 
$$
t \longmapsto \ip{v_-, t^{-1} \cdot C \cdot v_+} \in C^\infty(T)
$$
in two ways.  On the one hand, since $V$ has central character $\xi$, we have
$$
\ip{v_-, t^{-1} \cdot C \cdot v_+} = \xi(C) \psi'(t).
$$
On the other hand, noting that for $X \in \h$ we have $\ip{v_-, t^{-1} \cdot X \cdot v_+} =-\frac{\partial \psi'}{\partial X}(t)$ by \eqref{eq:gT1},
\begin{align*}
(C.\psi)(t) &= \ip{v_-, t^{-1} \cdot \left(\sum_{i=1}^r h_i^2 + 2\sum_{\alpha \in R_+} f_\alpha e_\alpha + 2h_{\rho^\vee} \right)
 \cdot v_+} \\
 &=\sum_{i=1}^r \ip{v_-, t^{-1} \cdot h_i \cdot h_i \cdot v_+} - 2\sum_{\alpha \in R_+} \alpha(t) \ip{f_\alpha \cdot v_-, t^{-1} e_\alpha \cdot v_+} + 2\ip{v_-,t^{-1} \cdot h_{\rho^\vee} \cdot v_+} \\
 &= \Delta(\psi')(t)  - 2\sum_{i \in I}\alpha_i(t) \psi'(t)-2\left(\frac{\partial \psi'}{\partial h_{\rho^\vee}}\right)(t) \\
 &= 2\left(H - \frac{\partial }{\partial h_{\rho^\vee}}\right)\psi'(t) 
\end{align*}
where we have used that $e_\alpha \cdot v_+ = 0$ for $\alpha$ not simple, and
$$
t^{-1} f_\alpha t= \ad(t^{-1})(f_\alpha) = \alpha(t)f_\alpha.
$$
Thus $\psi'$ is an eigenfunction of $H - \frac{\partial}{\partial h_{\rho^\vee}}$ with eigenvalue $\frac{1}{2}\xi(C)$.  But a term by term calculation gives the equality of differential operators
$$
(\rho(t))^{-1}\left(H - \frac{\partial}{\partial h_{\rho^\vee}}\right)\rho(t) = H - \frac{1}{2}(\rho^\vee,\rho^\vee).
$$
Thus the function $\psi(t) = \rho(t)^{-1} \psi'(t)$ is an eigenfunction of $H$ with eigenvalue $\frac{1}{2}\xi(C')$ where by \eqref{eq:HCC} $C' = C-(\rho^\vee,\rho^\vee)$ can be succinctly described as the element in $Z(\g)$ satisfying $HC(C')(\la) = (\la,\la)$.

Now the same proof as above shows that for each $z \in Z(\g)$ there is a differential operator $\D_z$ on $T$ so that $\psi$ is an eigenfunction of $\D_z$.  Note that the differential operator $\D_z$ only depends on the central character $\xi$, and not the other choices $(v_{\pm}, V_{\pm})$.  As $z$ varies, this collection of differential operators is the quantum Toda lattice.  They commute because $Z(\g)$ is commutative.
\end{proof}
We call the function of Proposition \ref{P:Toda}, an eigenvector corresponding to central character $\xi$.  Thus a Whittaker function with central character $\xi_{-\rho}$ has eigenvalue $HC(C')(-\rho+\rho) = 0$; that is, it is a solution to the Toda Hamiltonian.

The proof of Proposition \ref{P:Toda} essentially establishes the following reformulation of Theorem \ref{thm:Toda}.  Let $\D(T)$ denote the ring of differential operators on $T$.

\begin{theorem}\label{thm:Toda2}
There is an algebra embedding $\kappa_q: \Sym(\h)^W \to \D(T)$ such that $\kappa_q(\frac{1}{2}(C +(\rho,\rho))) = H$, and the symbol of $\kappa_q(z)$ is equal to $z$.
\end{theorem}

\subsection{Principal series representations}
Consider the space $W_\mu$ of holomorphic functions on $B_-U^{w_0}$ satisfying 
$$
f(bu) =\mu(b) f(u) \qquad \mbox{for $b \in B_-$}
$$
Since $B_-U^{w_0}$ is open in $G$, the universal enveloping algebra $U(\g)$ acts on $W_\mu$ in the usual way:
\begin{align*}
(X.f)(g) &= \left.\frac{d}{da} f(g\exp(aX))\right|_{a=0} \\
(t.f)(g) &= f(gt)
\end{align*}

%The minus sign: $(g.h.f)(x) = (g.(h.f))(x) = (h.f)(g^{-1} x) = f(h^{-1}g^{-1} x) = ((gh).f)(x)$.

%(f.h.g)(x) = (f.h)(xg^{-1}) = f(xg^{-1}h^{-1})

\begin{prop}
The space $W_\mu$ has infinitesimal character $\xi_{\mu}$.
\end{prop}
\begin{proof}
Let $z \in Z(\g)$.  Since $z$ is $G$-invariant: $g z g^{-1}=\ad(g)(z) = z$ for all $g \in G$, we have that $z.f = f.z$, where the right action of $U(\g)$ on $W_\mu$ is given by the formula:
\begin{align*}
(f.X)(g) &= \left. \frac{d}{da} f(\exp(aX)\,g)\right|_{a=0} 
\end{align*}
Let us write $z = x+y$ where $x \in \n_- U(\g)$ and $y = HC'(z) \in U(\h)$.  Then $f.x = 0$ so that $f.z = f.y = \mu(y)$, where we think of $\mu$ as an element of $\h^*$.  Thus $W_\mu$ has infinitesimal character $\xi_{\mu}$.
%The statement is local: it depends only on the properties of functions $f \in W_\mu$ in the neighborhood of each point $g \in UTw_0U$.
\end{proof}

\section{Whittaker functions as integrals over geometric crystals}
\label{sec:integral}
\subsection{Definition}
Let $\la: T \to \C^*$ be a character of $T$.  Define the integral function 
\begin{equation}
\label{eq:integral}
\psi_\la(t) = \int_{(X_t)_{> 0} \subset X_t} \la(\gamma(x)) e^{-\f(x)}\omega_t.
\end{equation}

Let $T_{>0} \simeq \R_{>0}^r$ be the totally positive part of $T$.  It is generated by the elements $\{\alpha_i^\vee(a) \mid a \in \R_{>0}\}$.

\begin{theorem}\label{thm:main}
The function $\psi_\la(t)$ is a Whittaker function on $T_{>0}$ with infinitesimal character $\xi_{\la-\rho}$.
\end{theorem}

The reason to only consider the integral for $t \in T_{>0}$ is to simplify issues of convergence: for $t \in T_{>0}$ the integral becomes a real integral.
%We first remark that we will only discuss the convergence of the integral for $t \in T_{>0} \simeq \R_{>0}^r$.  Thus, strictly speaking, $\psi_\la(t)$ will be a function on $T_{>0}$ (or an open subset of $T$ containing $T_{>0}$), rather than $T$ itself.

This formula is a rather elegant analogue of the formula for the character of an irreducible representation for $G$.  In that case, we have a summation over a crystal instead of an integral over a geometric crystal.  See Section \ref{sec:gelfand} for further discussion.

To prove Theorem \ref{thm:main}, we follow Rietsch \cite{Rie2} and Gerasimov-Kharchev-Lebedev-Oblezin \cite{GKLO} to express this integral as a matrix coefficient of dual Whittaker modules.  Theorem \ref{thm:main} was conjectured by Rietsch \cite{Rie1} without the language of geometric crystals, and she proved it in \cite{Rie2} for the case $\la = 1$.  The proof here is essentially the same as hers.  Theorem \ref{thm:main} was also established by Chhaibi \cite{Chh} using probabilistic methods, though I believe he checked only that it is an eigenfunction of the quantum Toda Hamiltonian, and not of the whole quantum Toda lattice. 

\begin{remark}
The approach of \cite{Rie2,GKLO}, and ours, remain valid for other families of integration cycles $\Gamma_t \subset X_t$ where the integrand has exponential decay in the infinite directions.  In particular, Rietsch \cite{Rie2} identifies a particular family of compact cycles.  Givental \cite{Giv} suggests taking a non-degenerate critical point of the decoration $f$ on $X_t$ and
taking the union of descending gradient trajectories of the function $\Re(f|_{X_t})$
with respect to a suitable Riemannian metric.  However, it does not seem easy to explicitly identify all the possible families of integration cycles.  Nevertheless, it seems reasonable to conjecture that as we vary the integration cycle we obtain {\it all} Whittaker functions with infinitesimal character $\xi_{\la-\rho}$.
\end{remark}

\begin{remark}
Since $\gamma(x) \in T_{> 0} \simeq \R_{> 0}^r$, instead of the factor $\la(\gamma(x))$ we could use the factor $e^{\ip{\log(\gamma(x)),h}}$ for $h \in \C^r$ to define $\psi_h(t)$.  Theorem \ref{thm:main} still holds in this setting.  Indeed, the functions $\psi_h(t)$ then become analytic in $h$.  See \cite{Chh}.
\end{remark}

\subsection{Convergence}
Let us begin by commenting on the convergence of the integral \eqref{eq:integral}, but only briefly.  By Theorem \ref{thm:posdec}, the decoration $f:X_{>0} \to \C$ is positive.  Indeed, as shown in \cite{BK2}, $f(x)$ is a positive sum of ratios of minors of $x$.  Since $x$ is totally nonnegative, it follows that $f(x) > 0$.  To show that the integral converges, in \cite{Rie2} and \cite{Chh} it is shown that the sets $\{x \in (X_t)_{> 0} \mid f(x) \leq M\}$ are bounded for any $M > 0$.  They obtain:

\begin{proposition}[\cite{Rie2,Chh}] \label{prop:convergence}
The function $e^{-\f(x)}$ has exponential decay in all directions of $(X_t)_{>0}$ for $t \in T_{>0}$ and thus the integral \eqref{eq:integral} converges for $t \in T_{>0}$.
\end{proposition}

In the rest of this section, there will be related integrals where parts of the integrands have been differentiated (arising from the action of $U(\g)$ on $W_\mu$).  This produces only extra rational factors in the integral and do not affect convergence.  We will thus not comment on convergence issues, and work only formally from now on.

Chhaibi's work \cite{Chh} contains a more serious treatment of analytic properties of this integral. 

\subsection{Whittaker vectors in $W_\mu$}
Let us define $\tgamma: U^{w_0} \to T$ by $\tgamma(u):= \gamma(\eta(u))$.  In the following we will regard functions on $U^{w_0}$ as elements of $W_\mu$ in the obvious way.

\begin{prop}
The function $f_+(u) = e^{-\chi(u)}$ is a $\n_+$-Whittaker vector of $W_\mu$.
% and a $T$-weight vector with weight
\end{prop}
\begin{proof}
We have 
$$
(e_i \cdot f_+)(u) = \left.\frac{d}{da} (u x_i(a))\right|_{a=0} = \left.\frac{d}{da} (e^{-a} f_+(u))\right|_{a=0} = -f_+(u).
$$
%Also 
%$$
%(t \cdot f_+)(u) = f_+(t^{-1} u) = f_+( t^{-1} u t t^{-1}) = e^{\chi(t^{-1} u t)} \mu(t^{-1})
%$$
\end{proof}

\begin{lemma}\label{lem:easy}
Suppose $u \in U$.  Let $uy_i(a) = bu'$ for $b \in B_-$ and $u' \in U$.  Then 
$$
\tau(u') = x_{i^*}(a) \tau(u) \qquad \text{and} \qquad \tgamma(u') = \gamma(b^{-1})\tgamma(u)
$$
where $i \mapsto i^*$ is the automorphism of $I$ induced by $-w_0$.
\end{lemma} 
\begin{proof}
Since $x_{i^*}(a) = (\bar w_0)^{-1} y_i(-a) \bar w_0$, we compute that $\eta(u') = (b^{-1} u y_i(a)) w_0 (x_{i^*}(a) \tau(u)) = b^{-1}\eta(u)$.
\end{proof}

\begin{prop}
The function $f_-(u) = \nu(\tgamma(u)) e^{-\chi(\tau(u))}$ is a $\n_-$-Whittaker vector of $W_\nu$.
\end{prop}
\begin{proof}
By Lemma \ref{lem:easy},
$$
f_-(uy_i(a)) =  f_-(bu') = \nu(b^{-1})\nu(\tgamma(u')) e^{-\chi_i(\tau(u'))} = \nu(\tgamma(u))e^{-\chi_i(\tau(u))} e^{-a}.
$$
Differentiating, we get $f_i \cdot f_- = -f_-$.
\end{proof}

\subsection{Pairing}
Suppose $V_+$ has central character $\xi_\mu$ and $V_-$ has central character $\xi_\nu$.  Then the existence of a $\g$-invariant pairing $\ip{\cdot,\cdot}:V_- \times V_+ \to \C$ implies that $\xi_{\mu} = \xi_{-w_0 \nu}$.  Note that this is satisfied if $\mu +\nu = -2\rho$, since then $\mu= w_0 (-w_0\nu+\rho) - \rho = w_0 \cdot (-w_0 \nu)$ so that $\xi_{\mu} = \xi_{-w_0\nu}$.

\begin{prop}\label{prop:pairing}
Suppose $f \in W_\nu$ and $f' \in W_\mu$ where $\nu+\mu = -2\rho$.  Then assuming it converges, the pairing 
$$
\ip{f,f'} = \int_{U^{w_0}_{>0}} f(u)f'(u)  \rho(\tgamma(u)) \omega_U
$$
is a $\g$-invariant pairing.
\end{prop}

There is an embedding of $U^{w_0}$ into the flag variety $\Fl = B_- \backslash G$ via $u \mapsto B_-\backslash B_-u$ (we use the left quotient here to match the choices of our principal series representations).  As a first step we establish that
\begin{lemma}
The rational $n$-form $\omega' = \omega(u)\rho(\tgamma(u))$, considered as a meromorphic $n$-form on $B_-\backslash G$, is $U$-invariant.
\end{lemma}

For $g \in G$, let $R_g: \Fl \to \Fl$ denote the isomorphism given by right multiplication by $g$.  We let $R_g^*$ denote the pullback of forms.

\begin{proof}
Let $u' = ux_i(a)$.  Suppose $u = x_{i_1}(a_1) \cdots x_{i_\ell}(a_\ell)$ and $i_\ell = i$.  
Then
$$
R_{x_i(a)}^*(\omega)(u) = \omega(u) \frac{a_\ell}{a_\ell + a}.
$$
But by Proposition \ref{prop:gamma} we also have $\tgamma(u') = \tgamma(u)\alpha_{i^*}^\vee(\frac{a_\ell+a}{a})$ and $\rho(\alpha_{i^*}^\vee(\frac{a_\ell+a}{a})) = \frac{a_\ell+a}{a}$, cancelling the above factor.
%By Lemma \ref{lem:Uaction},
%$\gamma(x') = \gamma(x)\alpha_i^\vee(1+a\ph_i(x))
%$, so $\rho(\gamma(x')) = \rho(\gamma(x))(1+a\ph_i(x))$.
%Also if $x = u_1w_0u_2$ where $u = x_{i_1}(a_1) \cdots x_{i_\ell}(a_\ell)$ and $i_\ell = i$ then $x' = u_1'w_0 u_2'$ where
%$u_1' = x_{i_1}(a_1 + a) \cdots x_{i_\ell}(a_\ell)$.  Thus
%$$
%L_{x_i(a)}^*(\omega')(x) = \omega'(x) \frac{a_1}{a_1+a} (1+a/a_1)
%$$
%since $\ph_i(x) = 1/a_1$ by Lemma \ref{lem:ph}.
\end{proof}

Thus the form $\omega'$ has no poles on the open Schubert cell $B_- \backslash B_-U$ but a double pole along the Schubert divisors, the irreducible components of the complement in $\Fl$.

The space $W_\mu$ can be identified with the space of smooth sections of the line bundle $L_{\mu} = \C_\mu \times_{B_-} G \to B_- \backslash G$ over the open subset $B_- \backslash B_-U^{w_0}$.  Since $U$ acts transitively on $B_- \cdot U$, we obviously have $\dim(\Gamma_{\rm hol}(B_- \backslash B_-U, L_\mu)^U) \leq 1$.  The restriction of these $U$-invariant sections to $B_-U^{w_0}$ can be identified with the functions $f \in W_\mu$ satisfying
$f(bu) = c\mu(b)$ for some constant $c$.  The canonical bundle of the flag variety $\Fl$ is known to be $G$-equivariantly isomorphic to $L_{-2\rho}$.  It then follows that under this isomorphism $\omega'$ can be identified with the the meromorphic section of $L_{-2\rho}$ on $\Fl$ which is identified with a function $f \in W_{-2\rho}$ satisfying
$$
f(bu) = c\rho(b)^{-2}
$$
for some constant $c$.

\begin{proof}[Proof of Proposition \ref{prop:pairing}]
%First note that by Proposition \ref{prop:canonicalequiv}
%$$\int_{X_{>0}} f(x)f'(x)  \rho(\gamma(x)) \omega_X = 
%\int_{U^{w_0}_{>0}} f(u)f'(u)  \rho(\gamma(u)) \omega_U,$$
%so we can just work on $U^{w_0}$, or the image of $U^{w_0}$ in $\Fl$.

Let $X \in \g$.  We first claim that 
\begin{equation}\label{eq:pullback}
(\exp(aX) \cdot f)(\exp(aX) \cdot f') \rho(\tgamma(u)) \omega_U =  R^*_{\exp(aX)} (f(u)f'(u)  \rho(\tgamma(u)) \omega_U)
\end{equation}
where on the left hand side we have the local $G$-action in the spaces $W_\nu, W_\mu$, and on the right hand side we have pullbacks of meromorphic forms.  (The equality is to be interpreted locally: for each $u \in U$, the LHS makes sense for sufficiently small $a$.)

To see this notice that the pointwise product $f(u)f'(u)$ can be thought of as an element of $W_{\nu+\mu} = W_{-2\rho}$.  The local $G$-action on this space is identified with the action of $G$ on sections of $L_{-2\rho}$, which can in turn be identified with the action of $G$ on meromorphic $n$-forms.  Under this identification, by the previous discussion, $f(u)f'(u) \in W_{-2\rho}$ can be identified with the meromorphic $n$-form $f(u)f'(u)  \rho(\tgamma(u)) \omega_U$ on $B_-\backslash G$.  Differentiating \eqref{eq:pullback}, we get
\begin{align*}
(X \cdot f) f' + f (X \cdot f') \rho(\tgamma(u)) \omega_U &= \left.\frac{d}{da}R^*_{\exp(aX)} (f(u)f'(u)  \rho(\tgamma(u)) \omega_U)\right|_{a=0}\\ &= d \circ i_{X}  (f(u)f'(u)  \rho(\tgamma(u)) \omega_U),
\end{align*}
where $i_{X}$ denotes the contraction with respect to the vector field on $B_-\backslash G$ given by $X \in \g$.  (We have used the Cartan formula $\L_X \omega = i_{X} d\omega + d \circ i_X \omega$, and the fact that our form is a top form.)  Thus
$$
\ip{X \cdot f, f'} + \ip{f, X \cdot f'} = \int_{\Gamma} d \circ i_{X}  (f(u)f'(u)  \rho(\tgamma(u)) \omega_U).
$$
We want to apply Stoke's theorem to the right hand side even though $\Gamma = U^{w_0}_{>0}$ is not compact.  To that end, suppose $\Gamma$ lies in a compactification $\bar \Gamma = \Gamma \sqcup \partial \Gamma$.  Then by Stoke's formula the right hand side is equal to $\int_{\partial \Gamma} i_{X}  (f(u)f'(u)  \rho(\tgamma(u)) \omega_U)$.   The form $(f(u)f'(u)  \rho(\tgamma(u))$ and with it, $i_{X}  (f(u)f'(u)  \rho(\tgamma(u))$, has exponential decay in all directions on $\Gamma$ and so is identically 0 on $\partial \Gamma$ (see Proposition \ref{prop:convergence}).  This argument can also be carried out by approximating $\Gamma$ by an increasing sequence $\Gamma_1 \subset \Gamma_2 \subset \cdots$ of open submanifolds which cover $\Gamma$.  See also \cite[p.20]{Rie2}.
\end{proof}

%\begin{lemma}
%Let us consider the function $\nu(\gamma(x))$ a function on $U$.  Suppose
%$zu = u'z'$ where $u' \in U$ and $z' \in B_-$.  Then 
%$$
%\nu(\gamma(u')) = \nu(\gamma(u)) \nu(z').
%$$
%\end{lemma}
%\begin{proof}
%\end{proof}

\subsection{Proof of Theorem \ref{thm:main}}
Let us compute the function $t \mapsto \rho(t)\ip{f_-, t \cdot f_+}$:
\begin{align*}
&\rho(t) \int_{U_{\geq 0}} f_-(u) (t \cdot f_+)(u) \rho(\tgamma(u)) \omega_U\\
&= \rho(t)\mu(t) \int_{U_{\geq 0}} f_-(u) f_+(t^{-1}ut) \rho(\tgamma(u)) \omega_U \\
&=\rho(t)\mu(t) \int_{U_{\geq 0}} (\nu+\rho)(\tgamma(u))e^{-\chi(\tau(u))} e^{-\chi(t^{-1}ut)} \omega_U \\
&=(\rho+\mu)(t)  \int_{U_{\geq 0}} (\nu+\rho)(\gamma(uw_0 \tau(u)))e^{-\f(t^{-1}uw_0 \tau(u))} \omega_U \\
&=(\rho+\mu)(t) (\nu+\rho)(t)\int_{(X_{t^{-1}})_{\geq 0}} (\nu+\rho)(\gamma(x)) e^{-\f(x)} \omega_{t^{-1}}\\
&=\int_{(X_{t^{-1}})_{\geq 0}} (\nu+\rho)(\gamma(x)) e^{-\f(x)} \omega_{t^{-1}}
\end{align*}
where we have used that $\nu+\mu+2\rho= 0$, $\gamma(x = (t^{-1}ut) t^{-1} w_0\tau(u)) = t^{-1}\gamma(uw_0\tau(u))$, and that the form $\omega$ on $U$ is invariant under conjugation by $t \in T$.

But by Proposition \ref{P:Toda}, the function
$$
\psi_{\la}(t) = \int_{(X_{t})_{\geq 0}} \la(\gamma(x)) e^{-\f(x)} \omega_{t}
$$
is a solution of the quantum Toda lattice with infinitesimal character $\xi_{\la-\rho}$.

\section{Whittaker functions as geometric analogues of Schur functions}\label{sec:schur}
In this section we assume $G = GL_{n+1}$.  While in the previous sections we have supposed that $G$ is semisimple, all the results extend naturally to this case.  Using $GL_{n+1}$ instead of $SL_{n+1}$ makes the combinatorics slightly more elegant. 

\subsection{Integrals over Gelfand-Tsetlin patterns}\label{sec:gelfand}
We fix the reduced word $$\i=({n},{n-1}, \ldots, 1,{n},{n-1},\ldots,2,\ldots,{n})$$ of $w_0$.  So for $n = 4$ we would have $\i = (4,3,2,1,4,3,2,4,3,4)$.  Let 
$$
u = x_\i(a_{n,n+1},a_{n-1,n+1},\ldots,a_{1,n+1},a_{n-1,n},a_{n-2,n},\ldots,a_{1,2})
$$
and
$$
t = \diag(t_{n+1},t_n,\ldots,t_1).
$$

\begin{prop}
Let $x = utw_0u' \in X$.  Then
$$
\gamma(x) = t\prod_{1 \leq i < j \leq n+1} \alpha^\vee_{i,j}(a_{i,j}).
$$
Thus $\gamma(x) = \diag(\gamma_1(x),\ldots,\gamma_{n+1}(x))$ where
$$
\gamma_i(x) = t_i \frac{\prod_{k=i+1}^{n+1} a_{i,k} }{\prod_{k=1}^{i-1} a_{k,i}}.
$$
\end{prop}
\begin{proof}
This just follows from Proposition \ref{prop:gamma}.
\end{proof}

We omit the proof of the following, which can be deduced from \cite{BZ}, or through a straightforward but lengthy calculation.  See Section \ref{sec:combcrystal} for an example.
\begin{prop}
Let $x = utw_0u' \in X$.  Then
$$
\f(x) = \sum_{1 \leq i < j \leq n+1} a_{i,j} + \sum_{1 \leq i < j \leq n+1} \frac{t_{j}}{t_{j-1}a_{i,j}} \frac{a_{i-1,j-1}}{a_{i-1,j}} \cdots \frac{a_{1,j-1}}{a_{1,j}}.
$$
\end{prop}

Define new variables
$$
z_{i,j} = \frac{t_j}{a_{1,j}a_{2,j}\cdots a_{n+1-i,n+1-i+j}}
$$
where we assume that $z_{n+1,j} = t_j$. Then 
$$
\f(x) = \sum_{1 \leq i < j \leq n+1} \frac{z_{n+2-i,j-i+1}}{z_{n+1-i,j-i}} +\frac{z_{n+1-i,j-i}}{z_{n+2-i,j-i}} = \sum_{1 \leq b \leq a \leq n} \frac{z_{a+1,b+1}}{z_{a,b}} + \frac{z_{a,b}}{z_{a+1,b}}
$$
and weight given by $\gamma_i(x) = \dfrac{\prod_{j=1}^{n+1-i} z_{n+1-i,j}}{\prod_{j=1}^{n-i}z_{n-i,j}}$.

For fixed positive parameters $t_j$ the transformation $a_{i,j} \mapsto z_{i,j}$ is an invertible monomial transformation, so our Whittaker function from Theorem \ref{thm:main} is
\begin{equation}\label{eq:tableau}
\psi_\mu(t) = \int_{\R_{>0}^{n(n+1)/2}} \left(\prod_{i=1}^{n+1} \gamma_i^{\mu_i}\right) \exp\left(-\sum_{1 \leq j \leq i \leq n} \frac{z_{i+1,j+1}}{z_{i,j}} + \frac{z_{i,j}}{z_{i+1,j}}\right) \prod_{1 \leq j \leq i \leq n}\frac{dz_{i,j}}{z_{i,j}}
\end{equation}
where $\mu = (\mu_1,\mu_2,\ldots,\mu_{n+1}) \in \Z^n$ is a character of $T$.  

If we make the substitutions $\la = (-\mu_{n+1},\ldots, -\mu_1)$, and $t_i = x_i$ then we see that one has
$$
\psi_\mu(t_1,t_2,\ldots,t_{n+1}) = \Psi^{n+1}_\la(x_1,x_2,\ldots,x_{n+1})$$ 
where $\Psi^{n}_\la(x)$ are the Whittaker functions studied by Corwin, O'Connell, Sepp\"{a}l\"{a}inen, and Zygouras \cite{COSZ}, and O'Connell, Sepp\"{a}l\"{a}inen, and Zygouras \cite[(2.11)]{OSZ}.  Essentially the same formula was studied by Givental \cite{Giv}, and Gerasimov, Kharchev, Lebedev, and Oblezin \cite{GKLO}.

The parameters $z_{i,j}$ can be put into a Gelfand-Tsetlin pattern of the form
$$
\begin{array}{ccccccccc} t_5 & &t_4 && t_3 &&t_2 &&t_1  \\
&z_{4,4} && z_{4,3} && z_{4,2} && z_{4,1}& \\
&& z_{3,3} && z_{3,2} && z_{3,1} && \\
&&&z_{2,2} && z_{2,1} &&&\\
&&&&z_{1,1} &&&&
 \end{array}
$$

Equation \eqref{eq:tableau} should thus be compared to the formula
$$
s_\lambda(x_1,x_2,\ldots,x_n) = \sum_T x^{{\rm weight}(T)}
$$
for a Schur function as a generating function of semistandard Young tableaux.  In this analogy, the summation is replaced by an integral, the shape $\lambda$ is replaced by the highest weight $t$ of a geometric crystal, the variables $x_i$ are replaced by the parameter $\mu$, and finally the condition that a tableau be semistandard is replaced by $e^{-f(x)}$ where the ``potential'' $f(x)$ discourages, rather than forbids, certain inequalities between the variables.

%The parameters $z_{i,j}$ can be put into a Gelfand-Tsetlin pattern of the form
%$$
%\begin{array}{ccccccccc} t_5 & &t_4 && t_3 &&t_2 &&t_1  \\
%&z'_{1,5} && z'_{1,4} && z'_{1,3} && z'_{1,2}& \\
%&& z'_{2,5} && z'_{2,4} && z'_{2,3} && \\
%&&&z'_{3,5} && z'_{3,4} &&&\\
%&&&&z'_{4,5} &&&&
% \end{array}
%$$
\subsection{Identities}
These Whittaker functions satisfy many integral identities reminiscient of Schur function identities.  The following analogue of the Cauchy identity is \cite[(3.21)]{COSZ}, and was first proved by Stade \cite{Sta}.
\begin{theorem}
Suppose $s > 0$ and $\lambda, \nu \in \C^n$, where $\Re(\la_i + \nu_j) > 0$ for all $i,j$.  Then
$$
\int_{\R_{>0}^n} e^{-s/x_n} \Psi_\nu^n(x) \Psi_\la^n(x) \prod_{i=1}^n \frac{dx_i}{x_i} = s^{-\sum_{i=1}^n(\nu_i+\la_i)} \prod_{i,j} \Gamma(\nu_i+\la_j).
$$
\end{theorem}
Note that for $\Psi^n_\la(x)$, the `shape' is $x$ and so the above formula is indeed an analogue of a summation over the shape.  Here the Gamma-function takes the place of the familiar factor $\frac{1}{1-\la_i \nu_i}$ in the Schur function identity.  In \cite{OSZ} the Whittaker analogue of the identity $\sum_\la s_\la(x) = \prod_i 1/(1-x_i) \prod_{i<j} 1/(1-x_ix_j)$ can also be found.

Define the ``Baxter operator'' \cite{OSZ, GLO} for $\la \in \C$
$$
Q_{\la}^n(x,y) = \left(\prod_{i=1}^n \frac{y_i}{x_i}\right)^\la \exp\left( -\sum_{i=1}^n \frac{y_i}{x_i} - \sum_{i=1}^{n-1} \frac{x_{i+1}}{y_i}\right).
$$
The following analogue of the Pieri rule is \cite[Corollary 2.2]{GLO}.
\begin{theorem}
For suitable $\gamma \in \C$ and $\la \in \C^n$, we have
$$
\int_{\R_{>0}^n} Q_\gamma(x,y) \Psi^n_\la(x) \prod_{i=1}^n \frac{dx_i}{x_i}= \prod_{i=1}^n \Gamma(\gamma-\la_i)  \Psi^n_\la(y).
$$
\end{theorem}
This is the an analogue of the generating function of Pieri rules over all homogeneous symmetric functions.  In this case, the product $\prod_{i=1}^n \Gamma(\gamma-\la_i)$ takes the place of a generating function $\sum_{k} h_k t^k$ of homogeneous symmetric functions.  The Baxter operator is thus the geometric analogue of adding a horizontal strip to a Young tableaux.

\begin{remark}
Just as the Cauchy and Pieri identities for Schur functions, and many other properties, follow from the Robinson-Schensted-Knuth algorithm, Corwin, O'Connell, Sepp\"{a}l\"{a}inen, and Zygouras \cite{COSZ} study Whittaker functions using the geometric RSK algorithm of Kirillov \cite{Kir} and Noumi-Yamada \cite{NY}.
\end{remark}

\begin{remark}
Chhaibi \cite{Chh} establishes many interesting properties of the functions \eqref{eq:integral}, including a geometric analogue of the Littlewood-Richardson rule.
\end{remark}

\section{Mirror symmetry for flag varieties}\label{sec:mirror}
The formula \eqref{eq:tableau} for the Whittaker function was discovered by Givental \cite{Giv} in the context of mirror symmetry for flag varieties.  From this point of view, the family $\hw: \X \to T$ equipped with volume forms $\omega_t$ and decoaration/superpotential $\f: \X \to \C$ is a ``mirror family'' to the Langlands dual flag variety $G^\vee/B^\vee$.  As we have remarked, the fibers $X_t \subset \X$ are themselves the complement in $G/B_-$ of an anticanonical divisor, and so can be thought of as open Calabi-Yau varieties.

There are no proofs in this section; we merely hope to connect the previous discussion to the literature on quantum Schubert calculus and mirror symmetry.

\subsection{Toda lattice}\label{sec:classicaltoda}
The Toda lattice of $G$ is the Hamiltonian integrable system on $T$ with Hamiltonian 
$$
H(t,h^*) = \frac{1}{2}(h^*,h^*)  - \sum_{i \in I} \alpha_i(t)
$$
where $(t,h^*) \in T \times \h^*$ is identified with the cotangent space $T^*(T)$.  For example, for $G = SL_n$ we have 
$$
H = \frac{1}{2} \sum_{i=1}^n p_i^2 - \sum_{i=1}^{n-1} e^{x_{i+1}-x_i}
$$
where to be compatible with the classical defintions, here we use coordinates $x_i$ on $\h$, rather than $T$.  This models $n$ particles traveling on the line with position $x_i$, momentum $p_i$ (with 0 total momentum).  The energy of the system is given by the kinetic energy $\frac{1}{2} \sum_{i=1}^n p_i^2$ and a potential where only adjacent particles interact.   

Let $e = \sum_{i\in I} f_i^\vee \in \g^\vee$ be a principal nilpotent.  Also identify $\h^*$ with $\h^\vee$.  We consider the map 
$$
(t,h^*) \longmapsto e + h^* + \sum_{i \in I} \alpha_i(t) e_i^\vee \in e + \h^\vee + \bigoplus_{i\in I} \C \cdot e_i^\vee \subset \g^\vee 
$$
sending the cotangent space $T^*(T)$ to a tridiagonal space of matrices.  We denote by $\A^\vee \subset \g^\vee$ the image of this map.  Kostant \cite{Kos:Toda} shows that the Toda Hamiltonian is essentially the Killing form $(.,.)^\vee$ of $\g^\vee$ restricted to $\A^\vee$, and the other integrals of motion are given by the map $\A^\vee \to \h^\vee/W$ arising from $\O(\h^\vee)^W \simeq \O(\h^\vee/W) \simeq \O(\g^\vee)^G \hookrightarrow \O(\g^\vee) \rightarrow \O(\A^\vee)$, where the first map is Chevalley's restriction theorem.  The images of the generators of $\Sym(\h)^W \simeq \O(\h^\vee/W)$ in $\O(\A^\vee)$ are the integrals of motion of the Toda lattice.

Under the isomorphism $\A^\vee \simeq T^*(T)$, we get an embedding $\kappa: \Sym(\h)^W \hookrightarrow \O(T^*(T))$, whose image consists of Poisson commuting elements.  This embedding is the quasi-classical limit of the quantum Toda lattice $\kappa_q: \Sym(\h)^W \hookrightarrow \D(T)$ of Theorem \ref{thm:Toda2}.  The quasi-classical limit is obtained by $\partial/\partial h \mapsto h^\vee$, for $h \in \h$.

For example, for $G = SL_n$, we have
\begin{equation}\label{eq:toda}
\A^\vee = \left\{ \left(\begin{array}{ccccc}
p_1&e^{h_2-h_1}&&\vphantom{\ddots} \\
1&p_2&e^{h_3-h_2}&\vphantom{\ddots} \\
&1&p_3&\ddots \\
&&\ddots&\ddots&e^{h_n-h_{n-1}} \\
&&&1&p_n
\end{array}\right) \right \}
\end{equation}
where $\{p_i\}$ is the basis of $\h^\vee$ dual to $\{h_i\}$ (cf. \eqref{eq:jacobi}).  The integrals of motion of the Toda lattice are the coefficients of the characteristic polynomial of this matrix.

\subsection{Cohomology of flag varieties}
We use the same notation as in previous sections, with the caution that objects for $G$ previously are now associated to the Langlands dual $G^\vee$.  

By Borel's Theorem the cohomology and equivariant cohomology of a flag variety have the ring presentations
$$
H^*(G/B, \C) \simeq \Sym(\h^*)/\ip{\Sym(\h^*)^W_+} \qquad \text{and} \qquad H^*_G(G/B, \C) \simeq \Sym(\h^*).
$$
where $\ip{\Sym(\h^*)^W_+}$ denotes the ideal generated by the positive degree elements of the invariants $W$-invariants $\Sym(\h^*)^W$. The Schubert cell decomposition $\{BwB/B \mid w \in W\}$ of $G/B$ gives a Schubert basis
$$
H^*(G/B, \C) \simeq \oplus_{w \in W} \C \cdot \sigma^w \qquad \text{and} \qquad H^*_G(G/B, \C) \simeq \oplus_{w \in W} \Sym(\h^*)^W \cdot \sigma^w_G,
$$
where $H^*_G(\pt,\C) \simeq \Sym(\h^*)^W$.  From our point of view, it is instructive to think of $\Sym(\h^*)$ as functions on $\h$: the $H^*_G(\pt,\C)$-module structure on $H^*_G(G/B, \C)$ comes from the projection $p:\h \to \h/W:= \Spec(\Sym(\h^*)^W)$.  Then $H^*(G/B, \C)$ is identified with functions on the scheme theoretic fiber of $p^{-1}(0)$.  The Schubert classes are certain distinguished functions on these spaces.

\subsection{Quantum cohomology of flag varieties}
The quantum cohomology rings $QH^*(G/B,\C)$ are defined using Gromov-Witten invariants, arising from the enumeration of rational curves in the flag variety.  We will not present the details here, but only some formal properties.  We have vector space isomorphisms
$$
QH^*(G/B,\C) \simeq \O(T^\vee) \otimes H^*(G/B,\C) \qquad \text{and} \qquad QH^*_G(G/B,\C) \simeq \O(T^\vee) \otimes H^*_G(G/B,\C)
$$
where (simple coroot) coordinates on $T^\vee$, denoted $\{q_i \mid i \in I\}$ are called quantum parameters, so that $\O(T^\vee) \simeq \C[q_i^{\pm 1}\mid i \in I]$.  (Often, quantum cohomology rings are defined without inverting the quantum parameters $q_i$.)  It is convenient, and part of the quantum cohomology setup, to identify $\h^*$ with the Lie algebra, or tangent space, of $T^\vee$, so that $ \O(T^\vee) \otimes \Sym(\h^*) \simeq \O(T^*T^\vee)$, the coordinate ring of the cotangent bundle of $T^\vee$.  We also have quantum Schubert bases
$$
QH^*(G/B,\C) \simeq \oplus_{w \in W} \O(T^\vee)\cdot \sigma_q^w \qquad \text{and} \qquad QH^*_G(G/B,\C) \simeq \oplus_{w \in W} \O(T^\vee)  \otimes \Sym(\h^*)^W \cdot \sigma_{q,G}^w
$$

%$$
% \left(\begin{array}{ccccc}
%x_1&-1&&\vphantom{\ddots} \\
%q_1&x_2&-1&\vphantom{\ddots} \\
%&q_2&x_3&\ddots \\
%&&\ddots&\ddots&-1 \\
%&&&q_{n-1}&x_n
%\end{array}\right)
%$$

Kim \cite{Kim}, following work of Givental-Kim \cite{GK} in type $A$, relate the ring structure of $QH^*(G/B,\C)$ with Toda lattices.  
\begin{theorem}[\cite{Kim}]\label{thm:Kim}
We have ring isomorphisms
$$QH^*(G/B,\C)\simeq \O(\A \times_{\h/W} \{0\})  \qquad \text{and} \qquad QH^*_G(G/B,\C) \simeq \O(\A)$$
where the $H^*_G(\pt,\C)$-module structure on $QH^*_G(G/B,\C)$ is given by the map $\A \to \h/W$.
\end{theorem}

The space $\A \times_{\h/W} \{0\}$ is the set of tridiagonal matrices where the integrals of motion of the Toda lattice vanish, often called the nilpotent leaf.  Thus the Schubert bases $\{\sigma^w_q\}$ should be thought of as functions on the nilpotent leaf $\A \times_{\h/W} \{0\}$ of the Toda lattice, and $\{\sigma^w_{q,G}\}$ as functions on the Toda lattice $\A$.  For example, for $G = PGL_3$, using \eqref{eq:toda} we would have that $QH^*(G/B,\C)$ is isomorphic to
$$
 \C[p_1,p_2,p_3,q_1^{\pm 1}, q_2^{\pm 1}]/\ip{p_1+p_2+p_3, q_1+q_2-p_1 p_2 - p_1 p_3 - p_2 p_3,p_1 p_2 p_3 - p_3 q_1 - p_1 q_2}
$$
where $q_1 = e^{h_2-h_1}$ and $q_2 = e^{h_3-h_2}$.

%The map $\kappa:\Sym(\h^*)^W \to \O(T^*T^\vee)$ is the {\it Toda} map. 
\begin{remark}
The proof of Theorem \ref{thm:Kim} utilizes heavily the formal properties of quantum cohomology, and remarkably few properties of rational curves in $G/B$.  The crucial fact is that $H_2(G/B,\C) \simeq \h \simeq \oplus_{i \in I} \C \cdot h_{\alpha_i}$ and the only rational curves of degree $h_{\alpha_i}$ in $G/B$ are the fibers of the projection map $G/B \to G/P_i$, where $P_i$ denotes a minimal parabolic subgroup.  In particular, there is exactly one rational curve of degree $h_{\alpha_i}$ through each point of $G/B$.
\end{remark}

\subsection{Mirror conjecture and quantum $D$-module}
The quantum cohomology ring $QH^*(M)$ is only part of the remarkable structure encoded in enumeration of rational curves in a (suitable) space $M$.  There is also a system of commuting differential equations called the quantum $D$-module.  We explain this now for the flag variety.

Just as the equivariant quantum cohomology ring of $G/B$ is a map $T^*T^\vee \to \h/W$, the quantum integrable system of the flag variety is an embedding
$$
\Sym(\h^*)^W \to \D(T^\vee)
$$
with quasi-classical limit given by the map $\Sym(\h^*)^W \to \O(T^*T^\vee)$ arising from the $H^*_G(\pt,\C)$-module structure of $QH^*_G(G/B,\C)$.  Here $\D(T^\vee)$ denotes the (non-commutative) ring of differential operators on $T^\vee$.

\begin{theorem}[\cite{Kim}]
The quantum cohomology $D$-module is given by the quantum Toda lattice of Theorem \ref{thm:Toda} for $G^\vee$.
\end{theorem}

Motivated by considerations from singularity theory, Givental \cite{Giv} proposed (in the general setting of a compact symplectic manifold $M$) that the quantum $D$-module had solutions which are stationary phase integrals over a conjectural ``mirror family'' $Y \to T^\vee$:
$$
\psi(t) = \int_{\Gamma_t \subset Y_t} e^{\F_t} \omega_t
$$   
where $Y_t$ is a family of possibly non-compact complex manifolds, $\F_t:Y_t \to \C$ is a family of holomorphic functions called the ``superpotential'', and $\omega_t$ a family of non-vanishing top-dimensional holomorphic forms.  As the family of real (again, possibly non-compact) middle-dimensional cycles varies, one hopefully obtains all the solutions to the quantum $D$-module.  In the setting of equivariant quantum cohomology, one is supposed to be able to produce arbitrary eigenfunctions of the quantum $D$-module.  Givental proved this for $G = GL_n$ in the non-equivariant setting and Joe and Kim \cite{JK} extended this to the equivariant setting.  Rietsch \cite{Rie1} then proposed a conjectural mirror family for arbitrary partial flag varieties, and proves it in the non-equivariant setting for the full flag variety \cite{Rie2}.  Our Theorem \ref{thm:main} proves Rietsch's conjecture in the equivariant setting.

In Givental's mirror conjecture, the critical points of the functions $\F_t$ recover the Lagrangian variety that is the spectrum of the quantum cohomology ring.  Rietsch \cite{Rie1} verified this for $G/B$ in the following

\begin{theorem}\label{thm:critical}
The set of critical points of the superpotential/decoration $f|_{X_t}:X^\vee_t \to \C$, as $t$ varies, is given by the centralizer subvariety
$$
(\X^\vee)^{{\rm critical}} = Z^\vee_e := \{x \in \X^\vee \mid \ad(x)(e) = e\}
$$
where $e = \sum_{i=1}^r f_i \in \g$ is the principal nilpotent element.
%$e = \sum_{i=1}^r (e_i^\vee)^* \in (\g^\vee)^*$ is the principal nilpotent element. %corresponding to our choice of character $\chi: U \to \C$.
\end{theorem}
The action $\ad(x)$ is the coadjoint action of $G^\vee$ on $\g \simeq (\g^\vee)^*$.  The variety $Z^\vee_e$ is related to Kim's presentation of $QH^*(G/B)$ by the map $\eta$.  The map
\begin{equation}\label{eq:adeta}
x \longmapsto \ad(\eta^{-1}(x))^{-1} \cdot e 
\end{equation}
maps $Z^\vee_e$ isomorphically onto $\A \times_{\h/W} \{0\} \simeq \Spec(QH^*(G/B,\C))$.  This map is an important part of Kostant's solution of the Toda lattice \cite{Kos:Toda}: it provides action coordinates for the nilpotent leaf.  In other words, the twist map is a part of the solution of the Toda lattice.  

The closure of the image $Z^\vee_e w_0 B^\vee_- \subset G^\vee/B^\vee_-$ of $Z^\vee_e$ in the flag variety is called the {\it Peterson variety}.  We refer the reader to Rietsch's work \cite{RieTP,Rie1} for further discussion of this.

\subsection{Quantum equals affine and Schubert bases}
Let $\Gr_G = G(\C((t)))/G(\C[[t]])$ denote the affine Grassmannian of $G$.  Here $\C((t))$ is the field of formal Laurent series, and $\C[[t]]$ is the subring of formal power series.  The affine Grassmannian also has a Schubert decomposition and we have a Schubert basis:
$$
H_*(\Gr_G,\C) \simeq \oplus_{x \in W_\af/W} \C \cdot \xi_x
$$
where $W_\af$denotes the affine Weyl group.

The affine Grassmannian is homotopy equivalent to the based loop group $\Omega K$, where $K \subset G$ denotes a maximal compact subgroup.  Thus the group multiplication of $\Omega K$ endows the homology $H_*(\Gr_G,\C)$ with a ring structure (indeed, a Hopf algebra structure).  This homology ring was studied by Ginzburg \cite{Gin} and Peterson \cite{Pet}, who showed:
\begin{theorem}\label{thm:Gin}
We have $H_*(\Gr_G,\C)_{\xi_{t_\la}^{-1}} \simeq \O(Z^\vee_e)$.
\end{theorem}
Here $\xi_{t_\la}$ is a Schubert class labeled by a translation element.  We refer the reader to \cite{LS1,LS2} for further details.  Peterson \cite{Pet} noticed the remarkable fact that the isomorphism $H_*(\Gr_G,\C)_{\xi_{t_\la}^{-1}} \simeq QH^*(G/B,\C)$ induced by \eqref{eq:adeta} is compatible with Schubert bases.  The following result is established by Lam and Shimozono \cite{LS1} following Peterson's work (see also Leung and Li \cite{LL}).

\begin{theorem}
We have a ring isomorphism $H_*(\Gr_G,\C)_{\xi_{t_\la}^{-1}} \to QH^*(G/B,\C)$ sending each affine Schubert class $\xi_x$ to some product $q^d \sigma^w_q$ of a monomial in the quantum parameters, and a quantum Schubert class.
\end{theorem}

It is convenient to think of $H_*(\Gr_G,\C)$ as more closely related to $Z^\vee_e$ and the mirror family, while $QH^*(G/B,\C)$ is more closely related to the nilpotent leaf of the Toda lattice $\A$.  This is because the centralizer subgroup $Z^\vee_e \subset G^\vee$ naturally appears in the geometric Satake correspondence underlying Theorem \ref{thm:Gin}, while the quantum parameters are explicit in the tridiagonal form $\A$.  So in some sense the twist map $\eta$ connects quantum with affine.  

Now, let us define the totally positive part of the centralizer variety $(Z^\vee_e)_{>0}:= Z^\vee_e \cap \X^\vee_{>0}$, where $\X^\vee_{>0} = T^\vee_{>0} \cdot U^\vee_{>0}$.  Surprisingly, we can use Schubert bases (thought of as functions on $Z^\vee_e$) to pick out the totally positive part, linking quantum and affine Schubert calculus back to the positivity discussions in previous sections.  The following result was established by Rietsch and myself \cite{LR}:

\begin{theorem}\
\begin{enumerate}
\item 
$(Z^\vee_e)_{>0} = \{x \in Z^\vee_e \mid \xi_w(z) > 0 \text{ for all } w \in W_\af/W\}$;
\item
the map $Z^\vee_e \to T^\vee$ restricts to a homeomorphism $(Z^\vee_e)_{>0} \to T^\vee_{>0} \simeq \R_{>0}^r$.
\end{enumerate}
\end{theorem}

Part (2) was conjectured by Rietsch \cite{RieTP} who established it in type $A$.  Combined with Theorem \ref{thm:critical}, it shows that $\F|_{(X_t)_{>0}}$ has a unique critical point for $t \in T_{>0}$.  Chhaibi \cite[Theorem 5.2.11]{Chh} proved the related result that each the decoration $f|_{X^\vee_t}$ has a unique minimum on $(X^\vee_t)_{>0}$, for $t \in T^\vee_{>0}$.

\end{document}